\documentclass[12pt,MSc]{muthesis_2023}
\usepackage{amsmath}
\usepackage{amsfonts}
\usepackage{amssymb}
\usepackage{lipsum}
\usepackage[utf8]{inputenc}
\usepackage{graphicx}
\usepackage{tikz}
\usepackage{float}
\usepackage{mathtools}

\usepackage{amsthm}
\usepackage{xpatch}
\xpatchcmd\swappedhead{~}{.~}{}{}
\swapnumbers % swaps number and name of the facts

\theoremstyle{plain}% default
\newtheorem{FactCounter}{dummy}[section]% defines the counter
\newtheorem{Theorem}[FactCounter]{Theorem} %
\newtheorem{Proposition}[FactCounter]{Proposition} %
\newtheorem{Lemma}[FactCounter]{Lemma} %
\newtheorem{Corollary}[FactCounter]{Corollary} %
\theoremstyle{definition}
\theoremstyle{remark}
\newtheorem{Example}[FactCounter]{Example} %
\DeclareMathOperator{\ad}{ad}
\DeclareMathOperator{\gl}{\mathfrak{gl}}

\DeclareMathOperator{\fsl}{\mathfrak{fsl}}
\DeclareMathOperator{\Sl}{\mathfrak{sl}}

\DeclareMathOperator{\im}{Im}

\DeclareMathOperator{\id}{id}
\DeclareMathOperator{\GF}{GF}
\DeclareMathOperator{\Der}{Der}
\DeclareMathOperator{\Ind}{Ind}

\usepackage{cite}

\usepackage{imakeidx}
\makeindex

\usepackage{listings}
\usepackage{hyperref}

\begin{document}
%This should be the original project title. Any changes should be agreed with your supervisor.
\title{Irreducible $L$-modules in Modular Lie Algebras}

%Put your name here
\author{Eun H. Park}

%It will say Department of Mathematics
\school{Mathematics}
\faculty{Science and Engineering}

\beforeabstract

The main objective of this project is to determine all irreducible modules of a given modular Lie algebra. In contrast to ordinary Lie algebras, modular Lie algebras require an additional structure known as the $p$-mapping. The minimal $p$-envelope of a modular Lie algebra is restrictable, and there exists a one-to-one correspondence between the induced modules and certain original modules. By exploiting the properties of induced modules, this project aims to decompose a modular Lie algebra $L$ into irreducible $L$-modules. Several examples will be presented to demonstrate how such decompositions can be achieved for specific modular Lie algebras.

\afterabstract

% The next part is optional; however it is a good place to thank your
% supervisor and the people responsible for providing computer support ;-)
\prefacesection{Acknowledgements}
I am deeply thankful to Sons Hub, the student support centre, for all their help and encouragement, and to my friends and colleagues for their constant support. I would also like to express my sincere gratitude to Dr. David Stewart for suggesting this topic for my main project and for his valuable guidance throughout my work.

% The next line is NOT optional and MUST appear
\afterpreface

% Finally, you can start writing your dissertation

\chapter{Introduction}

In Chapter 2, the concepts of restricted Lie algebras and restrictable Lie algebras are introduced. The notions of the $p$-mapping and $p$-semilinearity, together with some theoretical preliminaries relevant to the project, will be discussed in this chapter.

In Chapter 3, we examine the fact that every restricted Lie algebra possesses a restricted enveloping algebra. This chapter will cover the properties of $p$-envelopes, restricted enveloping algebras, and universal $p$-envelopes.

In Chapter 4, the concept of an induced $L$-module is introduced. This chapter will explore the one-to-one correspondence between induced modules and certain submodules of the original modules.

In Chapter 5, we consider representations of modular Lie algebras, which possess an invariant called a character—a linear form on the algebra. Given a character $S$, the construction of reduced enveloping algebras will be presented.

Finally, using these theoretical foundations, we will decompose each given modular Lie algebra $L$ into irreducible $L$-modules. Through detailed examples, we will demonstrate how all irreducible $L$-modules can be obtained via case-by-case analysis.

\chapter{Restricted Lie Algebras and Restrictable Lie Algebras}
\chaptermark{Restricted Lie Algebras} %shortened running header
We simply write the Lie multiplication $xy$ for any $x,y$ in a Lie algebra $L$, instead of $[x,y]$. The theorems and lemmas refer to \cite[Section 3, Chapter 2]{SF}, \cite[Chapter 2]{N}, \cite{F} and \cite{H}.

Let $L$ be a Lie algebra over $F$ of $\operatorname{char} p$. A mapping $[p]: L \rightarrow L$ defined by $x \mapsto x^{[p]}$ is called a \textit{$p$-mapping}\index{p-mapping} if
\begin{enumerate}\label{pmap1}
\item[(1)]
$\ad a^{[p]} = (\ad a)^p$ \; for all $a \in L$.
\item[(2)]
$(\alpha a)^{[p]} = \alpha^p a^{[p]}$ \; for all $\alpha \in F$ and $a \in L$.
\item[(3)]
$(a + b)^{[p]} = a^{[p]} + b^{[p]} + \sum^{p-1}_{i=1} s_i(a, b)$ where $(\ad(a \otimes X + b \otimes 1))^{p-1}(a \otimes 1) = \sum^{p-1}_{i=1} i s_i(a, b) \otimes X^{i-1}$ in $L \otimes_F F[X]$ for all $a, b \in L$.
\end{enumerate}
A pair $(L, [p])$ consisting of a Lie algebra $L$ and a $p$-mapping $[p]$ is called a \textit{restricted Lie algebra}\index{restricted Lie algebra}.
$f:L\rightarrow L$ is said to be \textit{$p$-semilinear}\index{$p$-semilinear} if $f$ satisfies the property $f(ax+y)=a^pf(x)+f(y)$ for all $x,y\in L$ and $a\in F$. Let $(L, [p])$ is a restricted Lie algebra over $F$. A subalgebra $H\subset L$ is called a \textit{$p$-subalgebra}\index{$p$-subalgebra} if $x^{[p]}\in H$ for all $x\in H$. An ideal $I$ of $L$ is called a \textit{$p$-ideal}\index{$p$-ideal} if $x^{[p]}\in I$ for all $x\in I$. Let $S\subset L$ be a subset of a restricted Lie algebra $(L, [p])$. Define $S_p$ as $\cap H$ where the intersection is over all $p$-subalgebras $H$ containing $S$. Then $S_p$ is a $p$-subalgebra and is the smallest $p$-subalgebra of $(L, [p])$ containing $S$. $S_p$ is called the \textit{$p$-subalgebra generated by $S$ in $L$}\index{$p$-subalgebra generated by $S$}. Let $(L_1, [p]_1)$ and $(L_2, [p]_2)$ be restricted Lie algebras over $F$. A mapping $f:L_1\rightarrow L_2$ is said to be \textit{restricted}\index{restricted homomorphism} or is called \textit{$p$-homomorphism}\index{$p$-homomorphism} if $f$ is a homomorphism of Lie algebras and $f(x^{[p]_1})={f(x)}^{[p]_2}$ for all $x\in L$. A representation $\rho:L\rightarrow \gl(V)$ is said to be \textit{restricted}\index{restricted representation} if $\rho(x^{[p]})={\rho(x)}^p$ for all $x\in L$.

The given restricted Lie algebra $L$ in the following proposition \cite[Chapter 1]{SF} is not necessarily a $p$-subalgebra of $(G, [p])$.
\begin{Proposition}\label{p1pmappingpsemilinear}
Let $L$ be a subalgebra of a restricted Lie algebra $(G, [p])$ and $[p]_1: L \rightarrow L$ be a mapping. Then the following statements are equivalent:
\begin{enumerate}
\item[(1)]
$[p]_1$ is a $p$-mapping on $L$.
\item[(2)]
There exists a $p$-semilinear mapping $f: L\rightarrow C_G(L)$ such that $[p]_1=[p]+f$.
\end{enumerate}
\begin{proof}
Assume that $[p]_1$ is a $p$-mapping on $L$. Consider $f: L\rightarrow G$ defined by $f(x)=x^{[p]_1}-x^{[p]}$. As $(\ad f(x))(y)=0$ for all $x,y\in L$, $f$ maps $L$ into $C_G(L)$. For $x,y\in L$ and $a\in F$,
\begin{align*}
f(ax+y)&=a^px^{[p]_1}+y^{[p]_1}+\sum_{i=1}^{p-1}s_i(ax,y)-a^px^{[p]}-y^{[p]}-\sum_{i=1}^{p-1}s_i(ax,y)\\
&=a^pf(x)+f(y).
\end{align*}
This shows that $f$ is $p$-semilinear.

Conversely, as $f$ is $p$-semilinear, it remains to show that the sum of any two elements satisfies the third property of the definition. That is, for $x,y\in L$,
\begin{align*}
(x+y)^{[p]_1}&=(x+y)^{[p]}+f(x+y)\\
 & =x^{[p]}+f(x)+y^{[p]}+f(y)+\sum_{i=1}^{p-1}s_i(x,y)\\
 & = x^{[p]_1}+y^{[p]_1}+\sum_{i=1}^{p-1}s_i(x,y).
\end{align*}
\end{proof}
\end{Proposition}

$L$ is restrictable if and only if there exists a mapping $[p]:L\rightarrow L$ which makes $L$ a restricted algebra. A Lie algebra $L$ is said to be \textit{restrictable}\index{restrictable} if $\ad (L)$ is a $p$-subalgebra of $\Der(L)$, that is, $(\ad x)^p\in \ad(L)$ for all $x\in L$.

\begin{Proposition}\label{L1restL2rest}
Let $f: L_1 \rightarrow L_2$ be a surjective homomorphism of Lie algebras. If $L_1$ is restrictable, then $L_2$ is restrictable.
\end{Proposition}
\begin{proof}
As $f$ preserves the structure of $L_1$ to $L_2$, $f(L_1)=L_2$ is restrictable.
\end{proof}

\subsection{The Jordan-Chevalley decomposition}

One of important decompositions in Lie theory is \textit{Jordan–Chevalley decomposition} which is stated in the below and proved in Theorem \ref{jordandecomp}. The following Lemma refers to  \cite[Lemma 2.3.1]{SF}

\begin{Lemma}\label{kerf=0}
Let $V$ and $W$ be $F$-vector spaces and $f:V\rightarrow W$ be a $p$-semilinear mapping. Then the followings hold:
\begin{enumerate}
\item[(1)]
$\dim_F V=\dim_F \ker(f) +\dim_{F^p} f(V)$.
\item[(2)]
If $\langle f(V) \rangle =W$ and $\dim_F W=\dim_F V$, then $\ker(f)=0$.
\end{enumerate}
\end{Lemma}

\begin{proof}
\begin{enumerate}
\item[(1)]
$W$ has the structure of an $F$-vector space via $\alpha \cdot w \coloneqq \alpha^p w$ for all $\alpha\in F$ and $w\in W$. The subspaces of $W$ under this structure are exactly the $F^p$-subspaces of $W$, so $\dim_F f(V)=\dim_{F^p} f(V)$. As $f:V \rightarrow W$ is linear with respect to the $\cdot$-structure, $V/\ker f \cong f(V)$ by the first isomorphism theorem. By considering the dimensions of the modules of the both sides, This shows $(1)$.

\item[(2)]
Since every $F^p$-basis of $f(V)$ is a generating set for the $F$-space $\langle f(V)\rangle$, $\dim_F\langle f(V)\rangle \leq \dim_{F^p} f(V)$. Let $\langle f(V) \rangle =W$ and $\dim_F W=\dim_F V$. Then
\begin{align*}
\dim_F W & = \dim_F V\\
& = \dim _F \ker (f) + \dim_{F^p} f(V) \\
& =\dim_F \ker(f) + \dim_F \langle f(V)\rangle \\
& = \dim_F \ker(f) + \dim_F W
\end{align*}
 and this shows $\dim_F \ker(f)=0$.
\end{enumerate}
\end{proof}

\begin{Lemma}\label{p-semilinear<f(v)>=V}
Let $f:V\rightarrow V$ be a $p$-semilinear. Then the following statements are equivalent.
\begin{enumerate}
\item[(1)]
$\langle f(V) \rangle =V$.
\item[(2)]
For every $v\in V$, there exist elements $\alpha_1,\cdots, \alpha_n$ in $F$ such that $v=\sum_{i=1}^n \alpha_i f^i(v)$.
\end{enumerate}
\end{Lemma}
\begin{proof}
It is clear that if (2), then $\langle f(V) \rangle =V$. Conversely, let $v\in V$ and define $U \coloneqq \sum_{i \geq 1} Ff^i(v)$. Then $U$ is an $f$-invariant subspace of $V$. $f$ induces a $p$-semilinear map $\bar{f}: V/U \rightarrow V/U$ defined by $x+U \mapsto f(x)+U$ for all $x\in V$, which also satisfies $\langle \bar{f}(V/U)\rangle =V/U$. Note that $v+U \in \ker(\bar{f})$. By Lemma \ref{kerf=0} (2), $\ker (\bar{f})=0$ and thus $v\in U$.
\end{proof}

Let $(L, [p])$ be a restricted Lie algebra over $F$. An element $x\in L$ is said to be \textit{semisimple}\index{semisimple element} if $x\in (Fx^{[p]})_p$. An element $x\in L$ is said to be \textit{toral}\index{toral} if $x^{[p]}=x$.

The following proposition refers to  \cite[Lemma 2.3.3]{SF}

\begin{Proposition}\label{ysemisimpleinFxp}
\begin{enumerate}
\item[(1)]
If $x$ and $y$ are semisimple and $xy=0$, then $x+y$ is semisimple.
\item[(2)]
If $x$ is semisimple, then $y$ is semisimple for every $y\in (Fx)_p$.
\end{enumerate}
\end{Proposition}

\begin{proof}
\begin{enumerate}
\item[(1)]
Define a $p$-subalgebra $V$ as $(Fx+Fy)_p$ and consider the $p$-mapping $[p]: V \rightarrow V$. Note that any $p$-mapping is $p$-semilinear. By the semisimplicity of $x$ and $y$, $x\in (Fx^{[p]})_p$ and $y\in (Fy^{[p]})_p$. This implies $x,y \in \langle V^{[p]} \rangle$. This shows $\langle V^{[p]}\rangle=V$.
\item[(2)]
If $x$ is semisimple, then $x^{[p]^i}$ is semisimple for every $i\in \mathbb{N}$. Let $y=\sum_{i\geq 0} \alpha_i x^{[p]^i}\in (Fx)_p$. Then $y$ is a sum of commuting semisimple elements. By (1), this implies that $y$ is semisimple.
\end{enumerate}
\end{proof}

\begin{Theorem}\label{x[p]ksemisimple}
Let $(L, [p])$ be a finite-dimensional restricted Lie algebra over $F$. Then, for every $x\in L$, there exists a positive integer $k\in\mathbb{N}$ such that $x^{[p]^k}$ is semisimple.
\end{Theorem}

\begin{proof}
The family $(x^{[p]^i})_{i\geq0}$ is linearly independent. Then there exists $k\geq 0$ and $\alpha_1,\cdots, \alpha_n \in F$ such that $x^{[p]^k}=\sum_{i=1}^n \alpha_i x^{[p]^{k+i}}$. This shows that $x^{[p]^k}$ is semisimple.
\end{proof}

\paragraph{The Jordan-Chevalley decomposition}
Let $F$ is an algebraically closed field. When $F$ is perfect, then the \textit{Jordan-Chevalley decomposition} of an endomorphism $g:V\rightarrow V$ is defined as the following. If $V$ is finite-dimensional, there exist two endomorphisms $S,N:V\rightarrow V$ where $S$ is semisimple and $N$ is nilpotent such that $g=S+N$ and $[S, N]=0$.

A \textit{perfect} field $F$ is a field where every irreducible polynomial has distinct roots. All algebraically closed fields are perfect.

\begin{Theorem}\label{jordandecomp}
Let $F$ be perfect and $(L, [p])$ be a finite-dimensional restricted Lie algebra over $F$. For any $x\in L$, there exists a uniquely determined elements $x_n, x_s\in L$ where $x_n$ is $p$-nilpotent and $x_s$ is semisimple satisfying $x=x_s+x_n$ and $x_sx_n=0$.
\end{Theorem}

\begin{proof}
By Theorem \ref{x[p]ksemisimple}, there exists $k\in \mathbb{N}$ such that $x^{[p]^k}$ is semisimple. Define $V$ as $(Fx^{[p]^k})_p$. Then by Proposition \ref{ysemisimpleinFxp}(2), every $v\in V$ is semisimple. $[p]$ is semilinear on the abelian subalgebra $V$. By Lemma \ref{p-semilinear<f(v)>=V}, $V=\langle V^{[p]}\rangle$. As $F$ is perfect, by Lemma\ref{kerf=0}(2), $V^{[p]}=\langle V^{[p]}\rangle$. This shows that there exists $x_s \in V$ such that $x^{[p]^k}={x_s}^{[p]^k}$. As $xV=0$, the element $x_n$, defined as $x-x_s$, is nilpotent and $x_sx_n=0$. It remains to show that such decomposition is unique. For any decomposition $x=x_s+x_n$ where $x_n$ is $p$-nilpotent and $x_s$ is semisimple such that $x_sx_n=0$, there exists $m\in \mathbb{N}$ such that ${x_s}^{[p]^m}=x^{[p]^m}\in (Fx)_p$. Then $x_s\in (F{x_s}^{[p]^m})_p\subset (Fx)_p$. As $x\in (Fx)_p$, $x_n=x-x_s\in (Fx)_p$. Let $x=x_n+x_s={x_n}'+{x_s}'$ be two Jordan-Chevalley decompositions of $x$. Then $x_s{x_s}'=x_n{x_n}'=0$ and $x_s-{x_s}'={x_n}'-x_n$ is both $p$-nilpotent and semisimple. This shows that $x_s-{x_s}'={x_n}'-x_n=0$.
\end{proof}

\begin{Corollary}
Let $(L, [p])$ be a finite-dimensional restricted Lie algebra over an algebraically closed field $F$. Consider the root space decomposition $L=H \oplus_{\alpha\in \phi} L_{\alpha}$ with respect to a Cartan subalgebra $H$. Then

\begin{enumerate}
\item[(1)]
If $h\in H$ is semisimple, then $\ad h\vert_{L_{\alpha}}=\alpha(h)\id L_{\alpha}$ and $\alpha(h)\in GF(p)$ for all toral $h\in H$.

\item[(2)]
If $h=h_n+h_s$ where $h_n$ is $p$-nilpotent and $h_s$ is semisimple, then $\alpha(h)=\alpha(h_s)$.
\end{enumerate}

\begin{proof}
\begin{enumerate}
\item[(1)]
Since $\ad h|_{L_{\alpha}}$ is semisimple, $\ad h|_{L_{\alpha}}$ is diagonalisable. This implies $\alpha(h)$ is the only eigenvalue of $\ad h|_{L_{\alpha}}$ and thus $\ad h|_{L_{\alpha}} = \alpha(h)\id_{L_{\alpha}}$. Let $h$ be a toral element. Then $\alpha(h)\id_{L_{\alpha}}=\alpha(h^{[p]})\id_{L_{\alpha}}=\ad h^{[p]}|_{L_{\alpha}}=(\ad h)^p|_{L_{\alpha}}=\alpha(h)^p\id_{L_{\alpha}}$. Simply, $\alpha(h)=\alpha(h)^p$. This shows that $\alpha(h)\in \GF(p)$.
\item[(2)]
For all $h\in H$, by (1), $(\ad h - \alpha(h_s)\id)|_{L_{\alpha}}=(\ad h-\ad h_s)|_{L_{\alpha}}=\ad h_n|_{L_{\alpha}}$, which is nilpotent. By the definition of nilpotent, $\alpha(h)=\alpha(h_s)$.
\end{enumerate}
\end{proof}
\end{Corollary}

\chapter{Restricted Enveloping Algebras and Universal $p$-Envelopes}\chaptermark{$p$-Envelopes}

Note that the theorems and lemmas refer to  \cite[Section 5, Chapter 2]{SF}.
\section{Restricted Enveloping Algebras for Restricted Lie Algebras}\chaptermark{Restricted Enveloping Algebras}

Let $L$ be a Lie algebra over a field $F$. Suppose $i: L\rightarrow U(L)^-$ is a homomorphism of Lie algebras of $L$ into the Lie algebra associated with the associative $F$-algebra $U(L)$. The pair $(U(L), i)$ is called \textit{universal enveloping algebra of $L$} if for every associative $F$-algebra $A$ and every homomorphism $f:L\rightarrow A^-$ of Lie algebras, there exists a unique associative homomorphism $\overline{f}:U(L)\rightarrow A$ such that the diagram
\begin{align*}
\begin{matrix}
U(L)&&\\
\uparrow i&\searrow&\overline{f}\\
L&\xrightarrow[]{f}&A
\end{matrix}
\end{align*}
commutes.

\begin{Theorem}\label{universalenvelopingalgebraunique}
Let $L$ be a Lie algebra. Then if $(U(L), i)$ and $(V(L), j)$ are universal enveloping algebras of $L$, then there exists a unique isomorphism $h: U(L) \rightarrow V(L)$ such that $h\circ i=j$.
\end{Theorem}

\begin{proof}
The universal property of $U(L)$ and $V(L)$ implies that there are two maps $f$ and $g$ such that the following diagrams
\begin{align*}
\begin{matrix}
U(L) && \\
i \uparrow &\searrow f &\\
L & \xrightarrow[j]{} & V(L)
\end{matrix}\;\;,
\quad&\quad
\begin{matrix}
V(L) && \\
j \uparrow &\searrow g &\\
L & \xrightarrow[i]{} & U(L)
\end{matrix}
\end{align*}
commute. This implies $(g\circ f)\circ i = g\circ j =i$ and $(f\circ g)\circ j=f\circ i=j$. That is, the following diagrams
\begin{align*}
\begin{matrix}
U(L) && \\
i \uparrow &\searrow &g\circ f \\
L & \xrightarrow[i]{} & U(L)
\end{matrix}\;\;,
\quad\quad&\quad
\begin{matrix}
V(L) && \\
j \uparrow &\searrow &f\circ g \\
L & \xrightarrow[j]{} & V(L)
\end{matrix}
\end{align*}
commute. On the other hand, the following diagrams
\begin{align*}
\begin{matrix}
U(L) && \\
i \uparrow &\searrow &\id_{U(L)} \\
L & \xrightarrow[i]{} & U(L)
\end{matrix}\;\;,
\quad\quad&\quad
\begin{matrix}
V(L) && \\
j \uparrow &\searrow &\id_{V(L)} \\
L & \xrightarrow[j]{} & V(L)
\end{matrix}
\end{align*}
commute. By the uniqueness in the definition, this implies that $g\circ f=\id_{U(L)}$ and $f\circ g=\id_{V(L)}$. It follows that $f$ and $g$ are isomorphisms.
\end{proof}
For any $k\geq 0$, the subspace $U_{(k)}$ of a universal enveloping algebra $U(L)$ is defined as $U_{(0)} := F1$ and $U_{(k)}:=\langle \left\{x_1\cdots x_l \;\;|\;\; l \leq k,\; x_j \in L \right\}\rangle + F1$. By the definition, $U(L)=\cup_{k\in \mathbb{N}_0} U_{(k)}$, $U_{(k-1)} \subset U_{(k)}$, and $U_{(k)}U_{(l)} \subset U_{(k+l)}$.

\begin{Lemma}\label{basisofU(L)}
Let $(e_i)_{i\in I}$ be an ordered basis of a Lie algebra $L$. Assume that there are a function $k: I \rightarrow \mathbb{N}$ and the families $(v_i)_{i\in I}$ and $(z_i)_{i\in I}$ such that for every $i\in I$, $e_i^{k(i)}=v_i+z_i$, $v_i\in U_{(k(i)-1)}$, $z_i\in C(U(L))$. Then the set $B := \{z^re^s\;|\; r,s\in N(I), \,s(i)<k(i),\;\; \text{for all}\;\; i\in I\}$ is a basis of $U(L)$.
\end{Lemma}

\begin{proof}
For any $t\in \mathbb{N}$, the set $B_t:=\{z^re^s \;|\; \sum_{i\in I}r(i)k(i)+|s| \leq t, s(i)<k(i)\}$ is a basis of $U_{(t)}$ by the following. As $z_i \equiv e_i^{k(i)} \mod U_{(k(i)-1)}$,
$$z_i^{r(i)}\equiv e_i^{k(i)r(i)}\mod U_{(r(i)k(i)-1)}.$$
This implies that
\begin{align*}
z^re^s&=\Pi_{i\in I}\,z_i^{r(i)}e_i^{s(i)}\\
&\equiv \Pi_{i\in I}\,e_i^{r(i)k(i)+s(i)} \mod U_{(t-1)}.
\end{align*}
To prove the uniqueness, let $(r,s)$ and $(r',s')$ be pairs such that $r(i)k(i)+s(i)=r'(i)k(i)+s'(i)$ for all $i\in I$. By the assumption, $0 \leq s(i)$ and $s'(i)<k(i)$ imply that $s(i)=s'(i)$ and $r(i)=r'(i)$ for all $i\in I$. By   , $B_t$ is linearly independent. For any $n\in N(i)$, there is a decomposition $n(i)=r(i)k(i)+s(i)$ where $0\leq s(i)<k(i)$ for all $i\in I$. This implies $B_t$ is a generating set of $U_{(t)}$. It follows that $B$ is a basis of $U(L)=\cup_{t\geq0}U_{(t)}$.

\end{proof}

\section{Restricted Enveloping Algebras and Universal $p$-Envelopes}
For restricted Lie algebras, there exists the universal enveloping algebra with the additional structure of restrictedness. Let $(L,[p])$ be a restricted Lie algebra. Then a pair $(u(L),i)$ consisting of an associative $F$-algebra $u(L)$ with its unity and a restricted homomorphism $i: L \rightarrow u(L)^-$ is called a \textit{restricted universal enveloping algebra} if given any associative $F$-algebra $A$ with its unity and any restricted homomorphism $f:L\rightarrow A^-$, there is a unique homomorphism $\bar{f}: u(L) \rightarrow A$ of associative $F$-algebras such that $\bar{f}\circ i=f$. By the universal property in Theorem \ref{universalenvelopingalgebraunique}, any two restricted universal enveloping algebras of $L$ are isomorphic.

\begin{Theorem}
Let $(L, [p])$ be a restricted Lie algebra. Then
\begin{enumerate}
\item[(1)]
There exists the restricted universal enveloping algebra.

\item[(2)]
Let $(u(L), i)$ be a restricted universal enveloping algebra and $(e_j)_{j\in J}$ be an ordered basis of $L$ over $F$. Then the elements $i(e_{j_1})^{s_1}, \cdots, i(e_{j_n})^{s_n}$ for $j_1 < \cdots < j_n, n\geq 0, 0 \leq s_k \leq p-1, 1\leq k \leq n$ form a basis of $u(L)$ over $F$.

By (2), $i:L\rightarrow u(L)$ is injective and $\dim_F u(L)=p^n$ if $\dim_FL=n$.
\end{enumerate}

\end{Theorem}

\begin{proof}
To apply Lemma \ref{basisofU(L)}, define $k(j)$ as $p$, $z_j$ as ${e_j}^p-{e_j}^{[p]}$ so $z_j \in C(U(L))$ and $v_i$ as ${e_i}^{[p]}$ so $v_i \in U_{(1)} \subset U_{(p-1)}$. Let $I := \sum_{j\in J}z_jU(L)$. $z_j \in C(U(L))$ means that all $z_j$ lie centrally in $U(L)$. This implies that $I$ is a two-sided ideal of $U(L)$. By Lemma \ref{basisofU(L)}, $I=\sum_{\sum r(j) \geq 1, 0\leq s(j)\leq p-1} Fz^re^s$ and $I \neq U(L)$. The elements $(e_{j_1}+I)^{s_1},\cdots, (e_{j_n}+I)^{s_n}$ for $j_1<\cdots<j_n, n\geq 0, 0 \leq s_k\leq p-1, 1\leq k\leq n$ form a basis of $U(L)/I$.

Define $u(L)$ as $U(L)/I$ and $i(x)$ as $x+I$, for all $x\in L$. To prove the theorem, it remains to show that $(u(L), i)$ is a restricted universal enveloping algebra. By \ref{p1pmappingpsemilinear}, the mapping $L\rightarrow U(L)$ defined by $x \mapsto x^p-x^{[p]}$ is $p$-semilinear. Then, for arbitrary $x=\sum \alpha_je_j\in L$,
\begin{align*}
(\sum \alpha_je_j)^p-(\sum\alpha_je_j)^{[p]}&=\sum {\alpha_j}^p({e_j}^p-{e_j}^{[p]})\\
&\equiv 0 \mod I.
\end{align*}
This implies $i(x)^p=i(x^{[p]})$ for all $x\in L$. Let $A$ be an associative algebra and $f:L\rightarrow A^-$ be a homomorphism such that $f(x)^p=f(x^{[p]})$ for all $x\in L$. Then $f$ has a unique extension $g:U(L)\rightarrow A$. $g(z_j)=g({e_j}^p-{e_j}^{[p]})=g(e_j)^p-g({e_j}^{[p]})=f(e_j)^p-f({e_j}^{[p]})=0$ for all $j$. This shows that $g(I)=0$. It follows that there exists a homomorphism $\bar{f}: U/I \rightarrow A$ such that the diagram
\begin{align*}
\begin{matrix}
L \;\;& \hookrightarrow \;\;& U(L) &\xrightarrow[]{\text{canonical}} & U(L)/I\\
  \;\;& f \searrow\;\; & \downarrow g & & \swarrow \bar{f}\\
  \;\;& \;\;& A & & 
\end{matrix}
\end{align*}
commutes. As $U(L)/I$ is generated by $i(L)$, $\bar{f}$ is uniquely determined by the equation $\bar{f}\circ i =f $.
\end{proof}

In $U(L)$, the Lie algebra $L$ is often identified with $i(L)$.

Let $L$ be any Lie algebra. A triple $(G, [p], i)$ consisting of a restricted Lie algebra $(G,[p])$ and a Lie algebra homomorphism $i: L\rightarrow G$ is called a \textit{$p$-envelope}\index{$p$-envelope} of $L$ if $i$ is injective and $(i(L))_p=G$. A $p$-envelope $(G, [p], i)$ is said to be \textit{universal}\index{universal} if it satisfies the following universal property: For every restricted Lie algebra $(H, [p]')$ and every homomorphism $f:L \rightarrow H$, there exists only one restricted homomorphism $g:(G, [p]) \rightarrow (H, [p]')$ such that $g\circ i=f$.

\begin{Theorem}[The existence of universal $p$-envelopes]
Every Lie algebra $L$ has a universal $p$-envelope $\widehat{L}$.
\end{Theorem}

\begin{proof}
Let $\widehat{L}$ be the $p$-subalgebra of $U(L)^-$ generated by $L$,  $(H, [p]')$ be any every restricted Lie algebra and $f:L\rightarrow H$ be a homomorphism. Consider $H$ as a subalgebra of $u(H)$ isomorphic to $H$. By the universal property of $U(L)$, there is an associative homomorphism $\bar{f}: U(L) \rightarrow u(H)$. We want to prove that $\bar{f}^{-1}(H)$ is a $p$-subalgebra containing $L$ and thus it contains $\widehat{L}$. Clearly, $L \subset \bar{f}^{-1}(H)$. If $x\in f^{-1}(H)$, then $f(x)=\bar{f}(x)\in H$ and $\bar{f}(x^p)=\bar{f}(x)^p={\bar{f}(x)}^{[p]}\in H$ in $u(H)$. This implies $x^p\in \bar{f}^{-1}(H)$. The homomorphism $\bar{f}: \widehat{L}\rightarrow H$ is an extension of $f$. As $\widehat{L}$ is generated by $L$ and $p$th powers of $L$, $\bar{f}$ is unique.
\end{proof}

\begin{Proposition}
Let $L$ be a Lie algebra. Then
\begin{enumerate}
\item[(1)]
Let $(\overline{L}, [p], i)$ be a $p$-enveloping of $L$. If $L$ is finite-dimensional, then $\overline{L}/C(\overline{L})$ is finite-dimensional.
\item[(2)]
If $\dim_FL$ is finite, then $L$ has a finite-dimensional $p$-envelope.
\end{enumerate}
\end{Proposition}

\begin{proof}
\begin{enumerate}
\item[(1)]
As $\overline{L}=L_p$, $\overline{L}^{(1)} \subset L$. $L$ is an ideal of $\overline{L}$. Consider the homomorphism $\varphi: \overline{L} \rightarrow \Der_F(L)$ defined by $\varphi(x)=(\ad x)|_L$. If $x \in \ker(\varphi)$, then $(\ad x)(L)=0$. $\ker (\ad x)$ is a $p$-subalgebra of $\overline{L}$. Then $x\in C(\overline{L})$ for all $x\in \ker(\varphi)$, so $\ker (\varphi)=C(\overline{L})$. By the first isomorphism theorem, $\dim_F \overline{L}/C(\overline{L})=\dim_F \im(\varphi)\leq \dim_F\Der_F(L)$, so $\dim_F \overline{L}/C(\overline{L})$ is finite.
\item[(2)]
Let $\widehat{L}$ be a universal $p$-envelope of $L$. Choose a subspace $V \subset C(\widehat{L})$ such that $C(\widehat{L})=V \oplus (C(\widehat{L})\cap L)$. By Proposition \ref{L1restL2rest}, $\widehat{L}/V$ is restrictable. $\widehat{L}/V$ contains $L$ isomorphically. By (1), $\dim_F \widehat{L}/V = \dim_F \widehat{L}/C(\widehat{L})+\dim_FC(\widehat{L})\cap L\leq \dim_F \widehat{L}/C(\widehat{L})+\dim_F L$ is finite. Then the $p$-subalgebra generated by $L$ in $\widehat{L}/V$ is a finite-dimensional $p$-envelope of $L$.
\end{enumerate}
\end{proof}

Let $L$ be a finite-dimensional. A $p$-envelope of a finite-dimensional Lie algebra is said to be \textit{minimal}\index{minimal} if its dimension is minimal among the dimensions of all $p$-envelopes of $L$.

\begin{Lemma}\label{vectorspaceconiLhasidealofG}
Let $(G, [p], i)$ be a $p$-envelope of $L$. Then any vector space $V$ containing $i(L)$ is an ideal of $G$.
\end{Lemma}

\begin{proof}
$VG \subset G^{(1)} \subset i(L)\subset V$.
\end{proof}

\begin{Proposition}\label{twopenvelopeshomexist}
Let $(G, [p], i)$ and $(G', [p]', i')$ be two $p$-envelopes of $L$. Then there exists a (not necessarily restricted) homomorphism $f:G \rightarrow G'$ such that $f\circ i = i'$.
\end{Proposition}

\begin{proof}
Let $\widehat{L} \supset L$ be a universal $p$-envelope of $L$. By the definition of the universal $p$-envelope, there exist $p$-homomorphisms $\widehat{i}$ extending $i$ and $\widehat{i}'$ extending $i'$. Then $\widehat{i}(\widehat{L})$ is a $p$-subalgebra of $G$ containing $i(L)$. Then $G={i(L)}_p=\widehat{i}(\widehat{L})$. This implies that $\widehat{i}:\widehat{L}\rightarrow G$ is surjective. By the definition of $\widehat{L}$, $\widehat{i} \mid_L=i$ is injective. Set $V$ as a subspace of $\widehat{L}$ containing $L$ such that $\widehat{i}\mid_V$ is an isomorphism of vector spaces. By Lemma \ref{vectorspaceconiLhasidealofG}, $V$ is a subalgebra of $\widehat{L}$. Define $x$ as $\widehat{i}\mid_V$, a restriction of a Lie algebra homomorphism. Then $x$ is an isomorphism of Lie algebras. $x^{-1}\circ i=x^{-1}\circ\widehat{i} \mid_L=\id_L$. From the commutative diagram
\begin{align*}
\begin{matrix}
 & & L & & \\
 i&\swarrow&\cap &\searrow&i'\\
G &\xrightarrow[\sim]{x^{-1}}&V&&G'\\
 \widehat{i}&\nwarrow&\cap&\nearrow&\widehat{i}'\\
 &&\widehat{L}&&
\end{matrix}
\end{align*}, $f:= \widehat{i}'\circ x^{-1}$ is a homomorphism from $G$ to $G'$ that satisfies $f\circ i=i'$.
\end{proof}

\begin{Proposition}\label{existpenvelope}
Let $L$ be a finite-dimensional and $(G, [p], i)$ and $(G', [p]', i')$ be two finite-dimensional $p$-envelopes of $L$. Suppose that $f:G \rightarrow G'$ is a homomorphism such that $f\circ i=i'$. Then
\begin{enumerate}
\item[(1)]
There exists an ideal $J\subset C(G)$ such that $G'=f(G)\oplus J$.

\item[(2)]
There exists a $p$-envelope $H\subset f(G)$ of $L$.
\end{enumerate}
\end{Proposition}

\begin{proof}
\begin{enumerate}
\item[(1)]
By Proposition \ref{twopenvelopeshomexist}, there exists a homomorphism $j': G' \rightarrow G$ such that $j'\circ i'=i$. Define $\mu$ as $f\circ j'$. Decompose $G'={G_0}'\oplus {G_1}'$ into its Fitting components with respect to $\mu$. As $\mu \circ i' = f\circ j' \circ i'=f\circ i=i'$, $\mu^k\circ i'=i'$ for any $k\in \mathbb{N}$. There exists $n$ such that ${G_0}'=\ker(\mu^n)$ and $\mu^n\circ i'=i'$. ${G_0}'$ is an ideal of $G'$. ${G_0}'$ intersects $i'(L)$ trivially. By Lemma \ref{vectorspaceconiLhasidealofG}, ${G_0}' \subset C(G')$. ${G_1}'=\mu({G_1}')\subset f(G)$. This implies $G'=f(G)+C(G)$. Set a direct complement $J$ of $f(G)$ which lies in $C(G')$.
\item[(2)]
As $G$ is restrictable, $f(G)$ is restrictable. Let $[p]''$ be a $p$-mapping on $f(G)$ and $H\subset f(G)$ be the $p$-subalgera of $f(G)$ which is generated by $i'(L)$. Then $(H, [p]'', i)$ is a $p$-envelope of $L$.
\end{enumerate}
\end{proof}

\begin{Theorem}
Let $L$ be a finite-dimensional Lie algebra. Then
\begin{enumerate}
\item[(1)]
Any two minimal $p$-envelopes of $L$ are isomorphic as ordinary Lie algebras.

\item[(2)]
If $(G, [p], i)$ is a finite-dimensional $p$-envelope of $L$, then there exist a minimal $p$-envelope $H \subset G$ and an ideal $J \subset C(G)$ such that $G= H\oplus J$ and $i(L)\subset H$.

\item[(3)]
A finite-dimensional $p$-envelope $(G, [p], i)$ is minimal if and only if $C(G)\subset i(L)$.
\end{enumerate}
\end{Theorem}

\begin{proof}
\begin{enumerate}
\item[(1)]
Let $(G, [p], i)$ and $(G', [p]', i')$ be minimal $p$-envelopes of $L$. By Proposition \ref{twopenvelopeshomexist} and Proposition \ref{existpenvelope}, there exist a homomorphism $f:G \rightarrow G'$ and an ideal $J\subset C(G')$ such that $f\circ i=i'$ and $G'=f(G)\oplus J$. By Proposition \ref{existpenvelope}(2), there exists a $p$-envelope $H\subset f(G)\subset G'$. As $G'$ is minimal, $H=G'$. By the inclusion relation, $H=f(G)=G'$, so $f$ is surjective. By the definition of the minimal $p$-envelope, $\dim_FG=\dim_FG'$. It follows that $f$ is bijective.

\item[(2)]
Let $(G', [p]', i')$ be a minimal $p$-envelope of $L$. By Proposition \ref{twopenvelopeshomexist}, $G=f(G')\oplus J$ where $f:G' \rightarrow G$ is a homomorphism with $f\circ i'=i$. By Proposition \ref{existpenvelope} (2), we set a $p$-envelope $H\subset f(G')$. By the minimality of $G'$, $H=f(G')$.

\item[(3)]
Let $G$ be minimal. $C(G)=C(G)\cap i(L) \oplus I$. Then $I$ is an ideal of $G$ with $I\cap i(L)=0$. By Theorem \ref{jordandecomp}, there exists a $p$-envelope $(G', [p]', i')$ with $G' \subset G/I$. By the minimality of $G$, $\dim_F G \leq \dim_F G' \leq \dim_F G/I$. This shows that $I=0$ and $C(G)\subset i(L)$.

Conversely, suppose $C(G)\subset i(L)$. By $(2)$, $G=H \oplus J$ where $H$ is minimal and $J \subset C(G)$. Note that $i(L)\subset H$. Then $J \cap i(L)=0$. As $J \subset C(G) \subset i(L)$, $J=J \cap i(L)=0$. Hence, $G=H$ and this implies that $G$ is minimal.
\end{enumerate}
\end{proof}

\chapter{Induced $L$-Modules}\chaptermark{Induced Representations of $L$}

The theorems and lemmas refer to  \cite[Section 6, Chapter 5]{SF}. A \textit{representation $\rho$ of a Lie algebra $L$} is a Lie homomorphism $\rho:L \rightarrow \gl(V)$. The vector space $V$ is called \text{the $L$-module corresponding to $\rho$}.

\begin{Theorem}\label{existextendreplsubmoduleGsubmodule}
Let $(G, [p], i)$ be a $p$-envelope of Lie algebra $L$. Suppose $\rho:L \rightarrow \gl(V)$ is a representation of $L$. Then there exists a representation $\widehat{\rho}: G \rightarrow \gl(V)$ extending $\rho$ such that every $L$-submodule of $V$ is a $G$-submodule.
\end{Theorem}

\begin{proof}
Let $\widehat{L}$ be the universal $p$-envelope of $L$, that is, $L \hookrightarrow \widehat{L}\subset U(L)$. By Proposition \ref{twopenvelopeshomexist}, a Lie algebra homomorphism $f: G \rightarrow \widehat{L}$ such that the diagram
\begin{align*}
\begin{matrix}
L & \hookrightarrow & \widehat{L}\\
i\downarrow &\nearrow f&  \\
G
\end{matrix}
\end{align*}
commutes. The representation $\rho$ of $L$ extends uniquely to a representation $\overline{\rho}$ of $U(L)$ which respects submodules of $L$. Then $\widehat{\rho}$ defined as $\overline{\rho}\circ f$ is an extension of $\rho$ to $G$. $\widehat{\rho}:G \rightarrow \gl(V)$ extending $\rho$ is a representation such that every $L$-submodule of $V$ is a $G$-submodule.

\end{proof}

We want to study modules of a given Lie algebra $L$ via some certain modules of subalgebra $H\subset L$. Throughout this section, we assume that an $F$-algebra $R$ is associative and has a unity 1 over the field $F$. Let $R$ and $S$ be $F$-algebras. A vector space $M$ is called \textit{$(R,S)$-bimodule}\index{$(R,S)$-bimodule}
if $M$ is both a left $R$-module and a right $S$-module and $r(ms)=(rm)s$ for all $r\in R, s\in S$ and $m\in M$. Let $S$ be an $F$-algebra, $M$ a right $S$-module and $N$ a left $S$-module. For a vector space $T$ and an $F$-bilinear mapping $f: M\times N \rightarrow T$, $f$ is said to be \textit{$S$-balanced}\index{$S$-balanced} if $f(ms, n)=f(m,sn)$ for all $(m,n)\in M\times N$ and $s\in S$. The pair $(T,f)$ is called a \textit{tensor product of $M$ and $N$}\index{tensor product} if, for any $F$-vector space $P$ and any balanced mapping $g:M\times N\rightarrow P$, there exists a uniquely determined linear mapping $\psi:T\rightarrow P$ such that $\psi\circ f=g$. By the universal property in the definition of a tensor product, the pair $(T,f)$ is determined up to isomorphisms. We denote $T=M\otimes_S N$ and $f(m,n)=m\otimes n$. The vector space $T$ is generated by the elements $m\otimes n$ where $m\in M$ and $n\in N$.

\begin{Lemma}\label{rmodulestructureontensorproduct}
Let $R$ and $S$ be $F$-algebras. Suppose $M$ is an $(R,S)$-module and $N$ is a left $S$-module. Then there exists a left $R$-module struture on $M\otimes_S N$ such that $r(m\otimes n) := (rm)\otimes n$ for all $(m,n)\in M\times N$ and $r\in R$.
\end{Lemma}

\begin{proof}
Let $r$ be an element of $R$. The mapping $f_r: M\times N \rightarrow M \otimes_S N$ defined by $f_r(m,n)=(rm)\otimes n$ is $F$-bilinear and balanced. Then there is a linear map $\psi_r: M\otimes_S N \rightarrow M \otimes_S N$ defined by $\psi_r(m\otimes n)=(rm)\otimes n$. For $v\in M\otimes_S N$ and $r\in R$, define $r\cdot v$ as $\psi_r(v)$. By the module properties, $\psi_r$ is uniquely determined.
\end{proof}

Suppose $M$ is a right $R$-module and $N$ is an $(R,S)$-bimodule. Then $M\otimes_S N$ is considered as a right $S$-module. If $M_1$ is a right $R$-module, $M_2$ is an $(R,S)$-bimodule, and $M_3$ is a left $S$-module, then we have a unique isomorphism $(M_1 \otimes_R M_2)\otimes_S M_3 \cong M_1 \otimes_R (M_2\otimes_S M_3)$ and the corresponding isomorphism is defined by $(m_1\otimes m_2)\otimes m_3\mapsto m_1\otimes(m_2\otimes m_3)$ for all $m_1\in M_1$, $m_2\in M_2$ and $m_3\in M_3$.

Let $H$ be a $p$-subalgebra of a restricted Lie algebra $(L, [p])$. As in enveloping algebras, for $S\in L^*$, $u(H, S|_H)$ can be embedded into $u(L,S)$. $u(L,S)$ becomes a free right $u(H, S|_H)$-module. The left and right multiplication of $u(L,S)$ by elements of $u(H, S|_H)$ provides $u(L,S)$ with the structure of an $(u(H, S|_H), u(H, S|_H))$-bimodule. Let $M$ be a left $H$-module with character $S|_H$. The \textit{induced $L$-module $\Ind^L_H(M,S)$ by $H$-module $M$}\index{Induced $L$-module by $H$-module} is defined as $u(L,S)\otimes_{u(H, S|_H)} M$, a left $L$-module with character $S$. The induced module $\Ind^L_H(M,S)$ is clearly a module by Lemma \ref{rmodulestructureontensorproduct}.

\begin{Proposition}\label{dimind}
Let $(L, [p])$ be a finite-dimensional restricted Lie algebra over $F$ and $S\in L^*$ be a linear form. Suppose that $H$ is a $p$-subalgebra of $L$ and $M$ is a finite-dimensional $H$-module with character $S|_H$. Then $\dim_F\Ind^L_F(M,S)=p^{\dim_F L/H}\cdot \dim_FM$.
\end{Proposition}

\begin{proof}
Let $\{e_1, \cdots, e_n\}$ be a basis for $L$ over $F$ such that $\{e_{m+1}, \cdots, e_n\}$ is a basis for $H$ over $F$. Define $\tau := (p-1, \cdots, p-1)$. Then $\{e^a \mid 0\leq a \leq \tau\}$ is a basis of $u(L,S)$ and $\{e^a\mid 0 \leq a \leq \tau-\sum_{i=1}^m (p-1)\epsilon_i\}$ is a basis of $u(H, S|_H)$ over $F$. This shows that
$$u(L,S)=\bigoplus_{0\leq a\leq r-(p-1)\sum_{i=m+1}^n \epsilon_i}e^au(H, S|_H).$$
By the definition of the induced module $\Ind_H^L(M, S)$ by $H$-module $M$, $F$-vector space isomorphisms
\begin{align*}
\Ind_H^L(M, S)&\cong\bigoplus_{0\leq a_i\leq p-1, 1\leq i \leq m}e^{a_i}u(H, S|_H)\otimes_{u(H, S|_H)}M\\
 &\cong\bigoplus u(H, S|_H)\otimes_{u(H, S|_H)}M \quad\quad(p^m \text{ summands})\\
 &\cong\bigoplus M \quad\quad(p^m \text{ summands})
\end{align*}
as $u(H, S|_H)\otimes_{u(H, S|_H)}M\cong M$. This shows that
\begin{align*}
\dim_F\Ind^L_H(M,S)&=p^m\cdot \dim_FM\\
& =p^{\dim_FL/H}\cdot \dim_FM.
\end{align*}
\end{proof}

\begin{Theorem}\label{IndandVhm}
Let $(L, [p])$ be a restricted Lie algebra over $F$ and $S\in L^*$. Suppose $V$ is an $L$-module with character $S$ and $H\subset L$ is a $p$-subalgebra. Assume that $M$ is an $H$-module with character $S|_H$ and $\psi:M\rightarrow V$ is an $H$-module homomorphism. Then there exists a homomorphism $\varphi: \Ind^L_H(M,S)\rightarrow V$ of $L$-modules defined by $\varphi(x\otimes m)=x\cdot\psi(m)$ for every $x\in u(L,S)$ and $m\in M$.
\end{Theorem}

\begin{proof}
It is clear that the $F$-bilinear mapping $f: u(L,S)\times M \rightarrow V$ defined by $f(x,m)=x\psi(m)$ is balanced with respect to $u(H, S|_H)$. Then there exists an $F$-linear mapping $\varphi: u(L,S)\otimes_{u(H, S|_H)}M\rightarrow V$ defined by $\varphi(x\otimes m)=x\cdot\psi(m)$. By definition of the $u(L,S)$-module structure on $\Ind_H^L(M,S)$, $\varphi$ is a homomorphism of $u(L,S)$-modules.
\end{proof}

If $V$ is irreducible and $\psi \neq 0$, then $\varphi$ is surjective by the following. $\varphi(\Ind_H^L(M,S))$ is a submodule of the irreducible module $V$, so $\varphi(\Ind_H^L(M,S))=V$ as $\psi$ is nonzero. This implies $\varphi$ is surjective.

\chapter{$S$-representations}\chaptermark{$S$-representations}
The theorems and lemmas refer to \cite[Section 3, Section 7, Chapter 5]{SF}.

\section{Character $S$}
\begin{Theorem}\label{Spxpx[p]=S}
Let $(L, [p])$ be a finite-dimensional restricted Lie algebra over an algebraically closed field $F$ and $\rho:L \rightarrow \gl(V)$ be an irreducible representation of $L$. Then there exists a linear form $S\in L^*$ such that $\rho(x)^p-\rho(x^{[p]})={S(x)}^p\id_V$ for all $x\in L$.
\end{Theorem}

\begin{proof}
As $V$ is finite-dimensional over $F$, every endomorphism ${\rho(x)}^p-\rho(x^{[p]})$ on $V$ has an eigenvalue $\alpha(x)\in F$. This shows $\left[{\rho(x)}^p-\rho(x^{[p]})-\alpha(x)\id_V, \rho(L)\right]=0$. $\ker({\rho(x)}^p-\rho(x^{[p]})-\alpha(x)\id_V)\neq 0$ is an $L$-submodule of $V$. By the irreducibility of $V$, $\ker({\rho(x)}^p-\rho(x^{[p]})-\alpha(x)\id_V)=V$. This implies that ${\rho(x)}^p-\rho(x^{[p]})-\alpha(x)\id_V$ is a zero map, i.e., ${\rho(x)}^p-\rho(x^{[p]})=\alpha(x)\id_V$. As $x \mapsto{\rho(x)}^p-\rho(x^{[p]})$ is $p$-semilinear, $\alpha$ is $p$-semilinear. Then $S(x) := \alpha(x)^{1/p}$ is a linear form in $L^*$ such that ${\rho(x)}^p-\rho(x^{[p]})={S(x)}^p\id_V$ for all $x\in L$.
\end{proof}

\begin{Example}\label{sl2isomorphismclasses}
Consider the retricted Lie algebra $\Sl(2,F)=Fe\oplus Ff\oplus Fh$ where $ef=h$, $he=2e$, $hf=-2f$, $h^{[p]}=h$, and $e^{[p]}=f^{[p]}=0$. Let $\rho: L\rightarrow \gl(V)$ be an irreducible representation of $L$ and $S$ be a linear form in $L^*$ such that ${\rho(x)}^p-\rho(x^{[p]})={S(x)}^p\id_V$ for all $x\in L$. Then
\begin{align*}
&{\rho(e)}^p={S(e)}^p\id_V,\\
&{\rho(f)}^p={S(f)}^p\id_V,\\
&{\rho(h)}^p-\rho(h)={S(h)}^p\id_V.
\end{align*}

\begin{enumerate}
\item[(a)] Let $S(e)=0$. Then ${\rho(e)}^p-\rho(e^{[p]})={S(e)}^p\id_V$ implies ${\rho(e)}^p=0$, i.e., $\rho(e)$ is nilpotent. The vector subspace $W$ of $V$ defined as $\{v\in V \mid \rho(e)(v)=0\}$ is nonzero, i.e., $W\neq (0)$. $W$ is invariant under $\rho(h)$. There exists a nonzero element $v\in W$ such that $\rho(e)(v)=0$, $\rho(h)(v)=\alpha v$ and $\alpha^p-\alpha = {S(h)}^p$. By induction on $i$,
\begin{align}
&\rho(h){\rho(f)}^i(v)=(\alpha-2i){\rho(f)}^i(v),\\
\label{rhoerhof}&\rho(e){\rho(f)}^{i+1}(v)= (i+1)(\alpha-i){\rho(f)}^i(v),\\
&{\rho(f)}^p={S(f)}^p\id_V.
\end{align}
$\sum_{i=0}^{p-1}F{\rho(f)}^i(v)$ is a nonzero submodule of $V$ and by the irreducibility of $V$, $V=\sum_{i=0}^{p-1}F{\rho(f)}^i(v)$. By the above relations, all the $\rho(f)^i(v)$ lie in different eigenspaces of $\rho(h)$.

\item[(a.1)] Let $S(f)\neq 0$. Then $\rho(f)^i(v)\neq 0$ for $0 \leq i \leq p-1$. $V= \oplus^{p-1}_{i=0}F{\rho(f)}^i(v).$

\item[(a.2)]
Let $S(f)=0$. For $k:=\min\{i \mid {\rho(f)}^i(v)=0\}$,
$$V=\sum_{i=0}^{p-1}F{\rho(f)}^i(v)= \oplus^{k-1}_{i=0}F{\rho(f)}^i(v).$$

\item[(a.2.1)] Let $S(h)\neq 0$. As $\alpha\notin \GF(p)$, $k=p$. If $k\neq p$, then $\alpha \in \GF(p)$. Assume $\alpha \not\equiv k-1 \mod p$. By the relation \ref{rhoerhof}, $0=\rho(e)\circ {\rho(f)}^k(v)=k(\alpha-k+1){\rho(f)}^{k-1}(v)$. This shows $\sum_{i=\alpha+1}^{p-1}F{\rho(f)}^i(v)$ is a proper submodule of $V$, which leads to a contradiction.

\item[(a.3)]
Let $S=0$. $\alpha=\dim_F V-1$.

\item[(b)] Let $S(e)\neq0$. By applying an automorphism to $L$,
\begin{align*}
e':= f+\lambda h-\lambda^2e, \quad\quad\quad\quad& f':=e, & h':=2\lambda e-h.
\end{align*}
where $\lambda$ is a solution of the equation $\lambda^2S(e)-\lambda S(h)-S(f)=0$. Then $S(e')=0$. Then the case (b) is dealt in the same way as case (a).
\end{enumerate}

Determine the number of isomorphism classes in each cases.

\item[(a)] Let $S(e)=0$.

\item[(a'.1)] Let $S(h)\neq0$. Then $\ker \rho(e)=Fv$ for some $v$. The eigenvalue $\alpha$ of $\rho(h)$ on $\ker \rho(e)$ distinguishes the isomorphism classes. The equation $X^p-X-{S(h)}^p=0$ has precisely $p$ solutions. This shows that there are $p$ nonisomorphic classes and each class is completely parametrised by $(S,\alpha)$.

\item[(a'.2)] Let $S(h)=0$.

\item[(a'.2.1)] Let $S(f)\neq0$. $\ker \rho(e)=Fv \oplus F{\rho(f)}^{\alpha+1}(v)$. Recall that $S(h)=0$ implies $\alpha \in \GF(p)$. The determinant of the restriction of $\rho(h)$ to $\ker\rho(e)$ is given by $\alpha(\alpha-2(\alpha+1))=-\alpha^2-2\alpha$. This scalar determines the isomorphism class completely. If $\alpha$ is a solution to the equation $X^2+2X+Y=0$ where $Y\in \GF(p)$, then $(-\alpha-2)$ is also a solution. The mapping ${\rho_1(f)}^i(v) \mapsto {\rho_2}^{i-\alpha-1}(w)$ is an isomorphism for $V$ corresponding to $\alpha$ and $W$ corresponding to $-\alpha-2$, so $V\cong W$. If $\alpha\neq -1$, then $X^2+2X+Y=0$ has 2 solutions. This implies that there are $(p+1)/2$ isomorphism classes.

\item[(a'.2.2)] Let $S=0$. Then there are $p$ isomorphism classes, determined by $\alpha=\dim_FV-1$.

\item[(b)] Let $S(e)\neq0$. Apply to the basis $\{e', f', h'\}$ and proceed as the above.

Consequently,
\item[(1)] If $S=0$ or ${S(h)}^2+4S(e)S(f)\neq0$, then there are exactly $p$ isomorphism classes.

\item[(2)] If $S\neq0$ and ${S(h)}^2+4S(e)S(f)=0$, then there are exactly $(p+1)/2$ isomorphism classes.

By considering the case-by-case analysis based on conditions concerning $S$, one can completely determine the irreducible representations.
\end{Example}

By Theorem \ref{Spxpx[p]=S}, the linear forms $S\in L^*$ are invariants of the isomorphism classes of irreducible $L$-modules. The following definition is generalisation of the cases we dealt previously. Let $(L,[p])$ be a restricted Lie algebra over $F$ and $S\in L^*$ be a linear form. A representation $\rho:L\rightarrow \gl(V)$ is called an \textit{$S$-representation}\index{$S$-representation} if ${\rho(x)}^p-\rho(x^{[p]})={S(x)}^p \id_V$ for all $x\in L$. $S$ is called \textit{the character of the representation}\index{character of the representation} or \textit{of the corresponding module}. By using linear forms $S \in L^*$, finite-dimensional of the universal enveloping algebra work with only finite-dimensional associative algebras instead of $U(L)$. Any $p$-representation is $S$-representation. More specifically, it is $S$-representation when $S=0$.

\section{Reduced Enveloping Algebras} \chaptermark{Reduced Enveloping Algebras}

By using the concept Reduced enveloping algebras, we study a new type of representations called $S$-representations in perspective of associative theory. Let $(L, [p])$ be a restricted Lie algebra and $S\in L^*$. A pair $(u(L,S), \iota)$ consisting of an associative $F$-algebra $u(L,S)$ with unity and a homomorphism $\iota: L \rightarrow u(L,S)^-$ such that ${\iota(x)}^p-\iota(x^{[p]})=S(x)^p1$ for all $x\in L$ is called an \textit{$S$-reduced universal enveloping algebra}\index{$S$-reduced universal enveloping algebra} if, for any associative $F$-algebra $A$ with unity and any homomorphism $f:L \rightarrow \overline{A}$ such that ${f(x)}^p-f(x^{[p]})={S(x)}^p1$ for all $x\in L$, there exists a unique homomorphism $\overline{f}: u(L,S) \rightarrow A$ such that $\overline{f}\circ \iota=f$. By Theorem \ref{universalenvelopingalgebraunique}, any two $S$-reduced universal enveloping algebras are isomorphisms.

\begin{Theorem}\label{existSredunienvalganditsbasis}
Let $(L, [p])$ be a restricted Lie algebra, $S\in L^*$, and $(e_i)_{i\in I}$ be an ordered basis of $L$. Define $\mathbb{N}(I):= \{f: I\rightarrow \mathbb{N} \mid f(i)=0 \text{   for all but finitely many }i\in I\}$. Then the $S$-reduced universal enveloping algebra $(u(L,S), \iota)$ exists and $\{{\iota(e)}^n\mid n\in \mathbb{N}(I), 0 \leq n(i) \leq p \text{ for all } i\in I\}$ is a basis of $u(L,S)$.
\end{Theorem}

\begin{proof}
Let $J$ be the ideal of $U(L)$ generated by $\{x^p-x^{[p]}-{S(x)}^p1 \mid x\in L \}$. Define $u(L,S)$ as $U(L)/J$ and $\iota: L\rightarrow u(L,S)$ as the restriction of the canonical projection $U(L)\rightarrow U(L)/J$. Let $A$ be an associative algebra with unity and $f:L\rightarrow A^-$ be a homomorphism such that ${f(x)}^p-f(x^{[p]})={S(x)}^p1$ for all $x\in L$. The mapping $f$ has a unique extension $U(f):U(L)\rightarrow A$ preserving unities. Then $x^p-x^{[p]}-{S(x)}^p1\in \ker(U(f))$ for all $x\in L$. This implies that there exists a mapping $\overline{f}:U(K)/J\rightarrow A$ such that the diagram
\begin{align*}
\begin{matrix}
L& \hookrightarrow &U(L) &\rightarrow& U(L)/J=u(L,S)\\
&f \searrow&\downarrow U(F)&\swarrow \overline{f}&\\
&&A&&
\end{matrix}
\end{align*}
commutes. $U(L)/J$ is generated by $\iota(L)$. Thus, $\overline{f}$ is uniquely determined by the equation $\overline{f}\circ \iota=f$.

To apply Lemma \ref{basisofU(L)}, set $k_i=p$, $z_i:= {e_i}^p-{e_i}^{[p]}-{S(e_i)}^p1\in C(U(L))$, and $v_i := {e_i}^{[p]}+{S(e_i)}^p 1\in U_{(1)}\subset U_{(p-1)}$. The $p$-semilinearity of the mapping $x\mapsto x^p-x^{[p]}-{S(x)}^p1$ implies that $J=\sum_{0\leq s(i)<p, 1 \leq |r|}Fz^re^s$ and $\{{\iota(e)}^n \mid n\in \mathbb{N}(I), 0 \leq n(i) <p\}$ is a basis of $u(L,S)$.
\end{proof}

Here, $L$ is identified with its image $\iota(L)$ in $u(L,S)$. The enveloping algebra $u(L)$ is $u(L,0)$, specifically, $u(L)$ is an $S$-reduced universal enveloping algebra when $S=0$. Note that $(u(L, S),\iota)$ depends on the $p$-map $[p]$ as the basis elements of $u(L, S)$ defined in the proof of the above theorem depends on $[p]$.

\begin{Corollary}
For every $S\in L^*$, there exists an irreducible $S$-representation of $L$.
\end{Corollary}

\begin{proof}
Let $I$ be a maximal left ideal of $u(L,S)$. For $x\in u(L,S)$, define $\rho(x)(u+I):=xu+I$ for all $u\in u(L,S)$. Then $\rho\mid_L: L \rightarrow \gl(u(L,S)/I)$ is an irreducible $S$-representation.
\end{proof}

The classical representation theory of associative algebras are used to study $S$-representations.

\begin{Example}
The irreducible representations of $\Sl(2)$. By Example \ref{sl2isomorphismclasses} (1), for the case when ${S(h)}^2+4S(e)S(f)\neq 0$, there are $p$ nonisomorphic irreducible $S$-representations and all such representations are of dimension $p$. Then by Theorem \ref{existSredunienvalganditsbasis}, $\dim_F u(\Sl(2), S)=p^3$, so $u(\Sl(2), S)$ is semisimple. It follows that every $S$-representations of $\Sl(2)$ is completely reducible.
\end{Example}

Let $A$ be an associative $F$-algebra. Suppose that $z, x_1,\cdots, x_n \in A$. For $t\in {\mathbb{N}_0}^n$, define $\{z,x;0\}:=z$ and $\{z,x;t\}:=[\cdots[z,x_1],\cdots, x_1], x_2],\cdots, x_2],\cdots x_n], x_n]$ where $[\cdots[z,x_1],\cdots, x_1]$ is multiplied $t_1$ times and $[\cdots[, \cdots, x_2],\cdots, x_2]$ is multiplied $t_2$ times and so on until $[\cdots[,\cdots x_n], x_n]$ multiplied $t_n$ times.

\begin{Lemma}\label{zxsexpension}
Let $z, x_1, \cdots, x_n$ be elements of an associative algebra $A$. Then $zx^s=\sum_{0\leq t\leq s}\binom{s}{t}x^{s-t}\{z,x;t\}$.

\end{Lemma}

\begin{proof}
We prove by induction on $n$. Let $R_y$ be right multiplication by $y$ in $A$ and $L_y$ be left multiplication by $y$ in $A$. Then $R_y=L_y-\ad y$ and $[L_y, \ad y]=0$. When $n=1$, it is true by    . For $n>1$ define $z':=z{x_1}^{s_1}\cdots {x_{n-1}}^{s_{n-1}}$ and $s':=(s_1,\cdots, s_{n-1})$. Assume $z'{x'}^{s'}=\sum_{0\leq t'\leq s'}\binom{s'}{t'}{x'}^{s'-t'}\{z, x'; t'\}$. Then 
\begin{align*}
zx^s=z'{x_n}^{s_n}&=\sum_{0\leq t'\leq s'}\binom{s'}{t'}{x'}^{s'-t'}\{z,x';t'\}{x_n}^{s_n}\\
&=\sum_{0\leq t'\leq s'}\binom{s'}{t'}{x'}^{s'-t'}\sum_{0\leq t_n\leq s_n}\binom{s_n}{t_n}{x_n}^{s_n-t_n}\{\{z,x';t'\}, x_n ; t_n\}\\
&=\sum_{0\leq t'\leq s'}\sum_{0\leq t_n\leq s_n}\binom{s'}{t'}\binom{s_n}{t_n}{x'}^{s'-t'}{x_n}^{s_n-t_n}\{z, x ; (t', t_n)\}\\
&=\sum_{0\leq t\leq s}\binom{s}{t}x^{s-t}\{z,x;t\}.
\end{align*}
\end{proof}

Note that $L^{(1)}$ is the derived subalgebra of $L$, that is, $LL$.

\begin{Lemma}\label{derivedsubalgebraandlinear}
Let $L$ be a linear Lie algebra of a finite-dimensional vector space $V$. Suppose $A\subset L$ is a Lie subalgebra.

\begin{enumerate}
\item[(1)]
If $A^{(1)}$ consists of nilpotent transformations and $F$ contains all eigenvalues for every $x\in A$, then there exists a common eigenvector $v\neq0$ such that $x(v)=\lambda(x)v$ for all $x\in A$.

\item[(2)]
Let $\lambda:A\rightarrow F$ be an eigenvalue function, i.e., $x-\lambda(x)\id_V$ is nilpotent for all $x\in A$. Assume that $\lambda(y)=0$ for all $y\in A^{(1)}$. Then $\lambda$ is linear.

\item[(3)]
If $A^{(1)}$ consists of nilpotent transformations, $F$ contains all eigenvalues for every $x\in A$, $A$ is an ideal of $L$ and $V$ is $L$-irreducible, then $A^{(1)}=0$, any $x\in A$ has a unique eigenvalue $\lambda(x)$ on $V$ and $\lambda:A \rightarrow F$ is linear.

\end{enumerate}
\end{Lemma}

\begin{proof}
\begin{enumerate}
\item[(1)]
Choose an $A$-irreducible subspace $W$ of $V$. By    , $x(v)=0$ for all $x\in A^{(1)}$ and $v\in W$. Let $y\in A$ be an arbitrary element. As $F$ contains all eigenvalues for every $x\subset A$, there is an eigenvector $w\neq 0$ in $W$ for $y$, that is, $y(w)=\lambda w$. Then $\{v\in W \mid y(v)=\lambda v\}\neq (0)$ is an $A$-submodule of $W$ as $A^{(1)}\mid_W=0$. By the $A$-irreducibility of $W$, $\{v\in W \mid y(v)=\lambda v\}=W$. This shows that every element of $W$ is an eigenvector for any $y\in A$.

\item[(2)]
As the assumptions of (1) are satisfied, let $u\neq 0$ be a common eigenvector, i.e., $x(u)=\lambda(x)u$ for all $x\in A$. As the left hand side is linear in $x$, the right hand side is also linear in $x$.

\item[(3)]
By   , $A^{(1)}=0$. Let $x\in A$ and $\lambda \in F$ be any eigenvalue of $x$. Then $[x^p, L]={(\ad x)}^p(L) \subset A^{(1)}=0$ and $\{v\in V \mid s^p(v)=\lambda^p v \}$ is a nonzero $L$-invariant subspace. By the irreducibility of $V$, $\{v\in V \mid s^p(v)=\lambda^p v \}=V$, i.e., $x-\lambda\id_V$ is nilpotent. Then $\lambda$ is the only eigenvalue of $x$ so by (2), $\lambda$ is linear.
\end{enumerate}
\end{proof}

Let $I$ be an ideal of a finite-dimensional restricted Lie algebra $L$ and $\lambda\in I^*$ such that $\lambda(I^{(1)})=0$. Define $L^{\lambda}:= \{x\in L \mid \lambda(xy)=0, \text{ for all } y\in I\}=I^{\perp}$. Then $L^{\lambda}$ is a $p$-subalgebra of $L$ by simply checking that $L^{\lambda}$ satisfies the condition of being $p$-subalgebra. Suppose $\{e_1,\cdots, e_m\}$ is a cobasis of $L^{\lambda}$. Then $L=L^{\lambda}\bigoplus \oplus_{i=1}^m Fe_i$. Given $S\in L^*$ and a finite-dimensional $L^{\lambda}$-module $M$ with the character $S\mid_{L^{\lambda}}$ such that $x\cdot m= \lambda(x)m$ for all $x\in I$ and $m\in M$, define $V:=\Ind^L_{L^{\lambda}}(M,S)$ and $V_{(j)}:=\sum_{0\leq s\leq \tau, |s|\leq j}Fe^s \otimes M$ for $\tau=(p-1, \cdots, p-1)$.

\begin{Lemma}\label{equationaboutyes}
\begin{enumerate}
\item[(1)]
There are $y_1,\cdots, y_m \in I$ such that $\lambda(y_ie_j)=\delta_{ij}$.

\item[(2)] For all $v\in M$, $(y_i-\lambda(y_i)1)\cdot e^s\otimes v \equiv s_i e^{s-\epsilon}\otimes v \mod V_{(|s|-1)}$.
\end{enumerate}
\end{Lemma}

\begin{proof}
\begin{enumerate}
\item[(1)] $U:=\sum_{i=1}^m Fe_i$. Define the bilinear form $B_{\lambda}(y,x):= \lambda(yx)$. Consider a linear mapping $\varphi: U \rightarrow I^*$ defined by $\varphi(x)=B_{\lambda}(\cdot, X)$. As $L^{\lambda}=I^{\perp}$, $\varphi$ is injective and thus the linear functionals $\varphi(e_j)$ are linearly independent. This implies that there are $y_1,\cdots, y_m\in I$ with $\varphi(e_j)(y_i)=\delta_{ij}$ for $1\leq i,j \leq m$.

\item[(2)]
Let $v\in M$. By Lemma \ref{zxsexpension}, $(y_i-\lambda(y_i)1)e^s=\sum_{0\leq t\leq s}\binom{s}{t}e^{s-t}\{y_i-\lambda(y_i)1, e; t\}$. For $t\neq 0$, $\{y_i-\lambda(y_i)1, e;t\}=\{y_i, e;t\}$ and $\{y_i, e;t\}\in I$. This implies that $\{y_i, e ; t\}\otimes v \in 1\otimes M=V_{(0)}$. 
\begin{align*}
(y_i-\lambda(y_i)1)e^s\otimes v&=\sum_{|t|\leq 1}\binom{s}{t}e^{s-t}\{y_i-\lambda(y_i)1, e; t\}\otimes v&\\
&\equiv e^s(y_i-\lambda(y_i)1)\otimes v+\sum_{j=1}^m s_je^{s-\epsilon_j}[y_i, e_j]\otimes v& \mod V_{(|s|-2)}\\
&\equiv s_ie^{s-\epsilon_i}\otimes v& \mod V_{(|s|-2)}
\end{align*}
\end{enumerate}
\end{proof}

Recall that $u(L,S)$ is a free right $u(L^{\lambda}, S\mid_{L^{\lambda}})$-module. For any $L^{\lambda}$-submodule $N\subset M$, there is a canonical embedding $\Ind^L_{L^{\lambda}}(N, S) \hookrightarrow \Ind^L_{L^{\lambda}}(M, S)$. 

\begin{Theorem}\label{submoduleisomorphictoinducedodule}
Let $W$ be an $L$-submodule of $\Ind^L_{L^{\lambda}}(M, S)$. Then there exists an $L^{\lambda}$-submodule $N$ of $M$ such that $W\cap(1\otimes M)=1\otimes N$ and $W\cong \Ind^L_{L^{\lambda}}(N, S)$. 

\end{Theorem}

\begin{proof}
Define $N$ as $\{m\in M \mid 1\otimes m \in W\}$. Then $N$ is an $L^{\lambda}$-submodule of $M$. $\Ind^L_{L^{\lambda}}(M, S)=\bigoplus_{0\leq s \leq \tau}Fe^s\otimes M$. This implies that $1\otimes N=W\cap (1\otimes M)$. Define $W_{(j)}:= \sum_{0\leq s\leq \tau, |s|\leq j}Fe^s\otimes N\subset \Ind^L_{L^{\lambda}}(M, S)$. Then $W_{(0)}=W\cap V_{(0)}$. Our claim is that $W\cap V_{(j)} \subset W_{(j)}$. We prove this by induction on $j$. Let $j\geq 1$ and assume $W\cap V_{(j-1)}\subset W_{(j-1)}$. Let $v\in W\cap V_{(j)}$. If $M=N\bigoplus \oplus^t_{k=1}Fm_k$, then $v=\sum^t_{k=1}\sum_{s\leq\tau, |s|\leq j}\alpha(k,s)e^s\otimes m_k$. Multiply $y_i-\lambda(y_i)1$ to $v$. Then, by Lemma \ref{equationaboutyes}, $(y_i-\lambda(y_i)1)\cdot v \equiv \sum_k\sum_{|s|\leq j}\alpha(k,s)s_ie^{s-\epsilon_i}\otimes_{1\leq i\leq m}m_k \mod V_{(j-2)}$. This implies that $(y_i-\lambda(y_i)1)\cdot v \in W\cap V_{(j-1)}\subset W_{(j-1)}$. By the definition of $W_{(j-1)}$, $\alpha(k,s)=0$ whenever $|s|=j$ for $j\geq1$, so $v=0$. $W\cap V_{(j)}\subset W_{(j)}$ for all $j\geq 0$, so $W= V_{(j)}\subset W_{(j)}$ for all $j\geq 0$. This implies that $W$ is the image of $ \Ind^L_{L^{\lambda}}(N, S)$ in $V$.
\end{proof}

\begin{Corollary}\label{indirredifforgirred}
$ \Ind^L_{L^{\lambda}}(M, S)$ is $L$-irreducible if and only if $M$ is $L^{\lambda}$-irreducible.

\end{Corollary}

\begin{proof}
It is clear that the $L$-irreducibility of $\Ind^L_{L^{\lambda}}(M, S)$ ensures that $M$ is $L^{\lambda}$-irreducible. Conversely, suppose that $M$ is irreducible. Let $W$ be an $L$-submodule of $\Ind^L_{L^{\lambda}}(M, S)$. By Theorem \ref{submoduleisomorphictoinducedodule}, there is an $L$-submodule $N$ of $M$ such that $W\cong \Ind^L_{L^{\lambda}}(N, S)$. By the irreducibility of $M$, $N=0$ or $N=M$. It follows that $W=0$ or $W=\Ind^L_{L^{\lambda}}(M, S)$. In other words, $\Ind^L_{L^{\lambda}}(M, S)$ is $L$-irreducible.
\end{proof}

For any $L$-module $V$, define $V^{\lambda}:=\{v\in V \mid y\cdot v = \lambda(y)v, \text{ for all } y\in I\}$. Note that $V^{\lambda}$ is an $L^{\lambda}$-submodule of $V$.

\begin{Corollary}\label{VinducVlambisoandchar}
Let $\rho: L \rightarrow \gl(V)$ be an irreducible representation of a finite-dimensional restricted Lie algebra. Suppose $I \lhd L$ is an ideal.

\begin{enumerate}
\item[(1)]
Assume $\rho$ is an $S$-representation and there exists a linear form $\lambda \in I^*$, $\lambda(I^{(1)})=0$ such that $V^{\lambda}\neq 0$. Then $V \cong \Ind^L_{L^{\lambda}}(V^{\lambda}, S)$ and $V^{\lambda}$ is an irreducible $L^{\lambda}$-module.

\item[(2)] If $F$ is algebraically closed and $I^{(1)}$ operates nilpotently on $V$, then there exists a character $S\in L^*$ and $\lambda\in I^*$ where $\lambda(I^{(1)})=0$ such that $V\cong \Ind^L_{L^{\lambda}}(V^{\lambda}, S)$.

\end{enumerate}
\end{Corollary}

\begin{proof}
\begin{enumerate}
\item[(1)]
By Theorem \ref{IndandVhm}, there is a homomorphism $\varphi: \Ind^L_{L^{\lambda}}(V^{\lambda}, S) \rightarrow V$ of $L$-modules such that $\varphi\neq0$ and $V$ is irreducible. Since $V$ is irreducibel and $\varphi \neq 0$, $\varphi$ is surjective and $\ker \varphi$ is an $L$-submodule of $\Ind^L_{L^{\lambda}}(V^{\lambda}, S)$ which intersects $1\otimes V^{\lambda}$ trivially. By Theorem \ref{submoduleisomorphictoinducedodule}, $\ker\varphi=0$. Thus, $\varphi$ is isomorphism. By Corollary \ref{indirredifforgirred}, $V^{\lambda}$ is irreducible.

\item[(2)]
By Theorem \ref{Spxpx[p]=S}, Lemma \ref{derivedsubalgebraandlinear}, and (1) that $\varphi$ is an isomorphism, (2) holds.
\end{enumerate}
\end{proof}

If $J\lhd L$ is an abelian ideal, then $I:= J+C(L)$ is an abelian $p$-ideal. By Corollary \ref{VinducVlambisoandchar}, $\lambda(y^p)-\lambda(y^{[p]})={S(y)}^p$ for all $y\in I$. If the $p$-mapping is trivial on the center of $L$, then $y^{[p]^2}=0$ for all $y\in I$ and $\lambda(y)=S(y)+S(y^{[p]})^{1/p}$ for all $y\in L$. Let $I\lhd L$ be an ideal. A linear mapping $\lambda \in I^*$ where $\lambda(I^{(1)})=0$ is called \textit{an eigenvalue function for an $L$-module $V$}\index{an eigenvalue function for an $L$-module $V$} if $V^{\lambda}\neq 0$. Let $S\in L^*$ and $\lambda\in I^*$ where $\lambda(I^{(1)})=0$. Define $A^{\lambda}_S$ as the set of isomorphism classes of irreducible $L$-modules with the character $S$ and the eigenvalue function $\lambda$ and $B^{\lambda}_S$ as the set of isomorphism classes of irreducible $L^{\lambda}$-modules with the character $S\mid_{L^{\lambda}}$ and the eigenvalue function $\lambda$.

\begin{Theorem}\label{ASlambisoBSlamb}
The mapping $\Gamma: A^{\lambda}_S \rightarrow B^{\lambda}_S$ defined by $[V] \mapsto [V^{\lambda}]$ is bijective.
\end{Theorem}

\begin{proof}
First, we want to prove that $\Gamma$ is well-defined. If $V$ is irreducible with the character $S$, then by Corollary \ref{VinducVlambisoandchar}, $V^{\lambda}$ is irreducible with the character $S\mid_{L^{\lambda}}$, so $\Gamma$ is well-defined. Define an inverse mapping $\theta: B^{\lambda}_S \rightarrow A^{\lambda}_S$ such that $[M]\mapsto [\Ind^L_{L^{\lambda}}(M, S)]$. By Corollary \ref{indirredifforgirred}, $\theta$ is well-defined. Let $M$ be irreducible $L^{\lambda}$-module with character $S\mid_{L^{\lambda}}$. Define $V:=\Ind^L_{L^{\lambda}}(M, S)$. Then $V$ is irreducible and $1\otimes M \subset V^{\lambda}$. By Corollary \ref{VinducVlambisoandchar}, $V \cong \Ind^L_{L^{\lambda}}(V^{\lambda}, S)$. This implies that $V^{\lambda}=1\otimes M$ and thus $\Gamma \circ \theta([M])=[1\otimes M]=[M]$. Conversely, assume that $V$ is an irreducible $L$-module with character $S$. By Corollary \ref{VinducVlambisoandchar}, $V\cong \Ind^L_{L^{\lambda}}(V^{\lambda}, S)$. This shows that $[V]=\theta(\Gamma([V]))$.
\end{proof}

\chapter{Irreducible $L$-Modules}
We want to determine irreducible representations of a modular Lie algebra $L$. The following examples refer to \cite[Section 9, Chapter 5]{SF}.

\begin{Example}
Let $L=Fe\oplus Fh\oplus Ff\oplus Fu\oplus Fv$ where $he=2e$, $hf=-2f$, $hu=u$, $hv=-v$, $ef=h$, $eu=0$, $ev=u$, $fu=v$, $fv=0$, and $uv=0$. $L$ is built by $\Sl(2)$ and its standard two-dimensional module $I=Fu\oplus Fv$. Note that $I$ is an abelian ideal of $L$. Let $V$ be an irreducible $L$-module and $F$ be an algebraically closed. By Lemma \ref{derivedsubalgebraandlinear}, there exists a linear eigenvalue function $\lambda: I \rightarrow F$ such that $\rho(x)-\lambda(x)\id_V$ is nilpotent for all $x\in I$.

Assume that $\lambda=0$. Then, by Lemma \ref{equationaboutyes}, $L^{\lambda}=L$. $V^{\lambda}\neq0$ is an $L$-submodule. By the irreducibility of $V$, $V=V^{\lambda}$. This implies $\rho(x)=0$ for all $x\in I$ and thus $\rho$ is an irreducible representation of $\Sl(2,F)$.

If $\lambda\neq0$, then $L^{\lambda}=F({\lambda(v)}^2e+\lambda(u)\lambda(v)h-{\lambda(u)}^2f)\oplus Fu\oplus Fv$. $(L^{\lambda})^{(1)}\subset Fu+Fv$. As $I=Fu+Fv$ is an abelian ideal, $L^{\lambda}$ is solvable.
\end{Example}

By Corollary \ref{VinducVlambisoandchar} and Theorem \ref{ASlambisoBSlamb}, the irreducible $L$-modules biject to the irreducible $L^{\lambda}$-modules. Hence, all irreducible $L$-modules are determined by $\Sl(2,F)$ and a certain solvable subalgebra of $L$.

The above example is for observation. With the above theoretical observations, the following section focuses on finding all irreducible $L$-modules for given modular Lie algebras. 

\section{The Isomorphism Classes of $S$-representations}

Every irreducible representation $\rho:L \rightarrow \gl(V)$ is uniquely determines a character $S\in L^*$. Every $S\in L^*$ has a character of some irreducible representation. From this, we describe the isomorphism classes of $S$-representations. For given $S$, determine that an ideal $I$ is maximal among all those ideals $J$ such that $J^{(1)}$ operates nilpotently on $V$ and find all those $\lambda\in I^*$ such that $V^{\lambda}\neq0$. By the above results, any irreducible $L$-module $V$ is induced by $V^{\lambda}$ and $V^{\lambda}$ is an irreducible $L^{\lambda}$-module.

\begin{Example}\label{hxisoclasses}
Consider $L:= Fh\oplus Fx$ where $hx=x$, $h^{[p]}=h$ and $x^{[p]}=0$, the unique two-dimensional restricted Lie algebra. Let $\rho:L \rightarrow \gl(V)$ be an irreducible representation. Note that any irreducible representations of $L$ are in bijection to characters $S\in L^*$. Thus, there is a unique character $S\in L^*$.

\begin{enumerate}
\item[(a)] Suppose $S(x)=0$. Note that ${\rho(x)}^p=\rho(x^{[p]})+{S(x)}^p\id_V=0$ as $x^{[p]}=0$ and $S(x)=0$. Then $I:= Fx$ is an ideal of $L$ which acts nilpotently on $V$. By the irreducibility of $V$, $\rho(x)=0$. This shows that $V$ is an irreducible $L/Fx$-module. $V$ is one-dimensional. $V=Fv$ where $x \cdot v=0$ and $h\cdot v=\alpha v$ for some $\alpha\in I$. To determine $\alpha$, note that $\rho(h)=\alpha\id_V$. ${S(h)}^p\id_V={\rho(h)}^p-\rho(h^{[p]})={\rho(h)}^p-\rho(h)=(\alpha^p-\alpha)\id_V$. This implies $\alpha$ is a solution of the equation $X^p-X={S(h)}^p$ in $F$. For $S$, there are $p$ nonisomorphic irreducible modules determined by the eigenvalue of $\rho(h)$ and they are all one-dimensional.
\begin{align*}
V=Fv,\;\;\quad\text{where} \; \quad& x \cdot v=0,\quad h\cdot v=\alpha v \;\text{  for some }\alpha\in I\\
&\text{where $\alpha$ is a solution of $X^p-X={S(h)}^p$.}
\end{align*}

\item[(b)]
Suppose $S(x)\neq 0$. Then ${\rho(x)}^p=\rho(x^{[p]})+{S(x)}^p\id_V={S(x)}^p\id_V$ as $x^{[p]}=0$. ${(\rho(x)-S(x)\id_V)}^p={\rho(x)}^p-{S(x)}^p\id_V=0$. This shows that $S(x)$ is the only eigenvalue of $\rho(x)$. The linear map $\lambda=S|_{Fx}\in {(Fx)}^*$ is the only eigenvalue function for which $V^{\lambda}\neq(0)$. $L^{\lambda}=Fx$ and $V^{\lambda}$ is $L^{\lambda}$-irreducible. As any irreducible module of an abelian ideal is one-dimensional, this implies that the dimension of $V^{\lambda}$ is one. For a given $S$, $V^{\lambda}$ is uniquely determined by
\begin{align*}
V^{\lambda}=Fv, \;\;\quad\text{where} \;\quad x \cdot v=S(x)v.
\end{align*}
By Theorem \ref{ASlambisoBSlamb}, $V$ is uniquely determined by $S$, that is,
\begin{align*}
V=\Ind^L_{L^{\lambda}}(V^{\lambda}, S) \quad\text{where $L^{\lambda}= Fx$, $V^{\lambda}=Fv$ such that $x \cdot v=S(x)v$}.
\end{align*}

By Proposition \ref{dimind}, $\dim_FV=p$.
\end{enumerate}
\end{Example}

\begin{Example}
Let $L:= Ft\oplus Fx\oplus Fy \oplus Fz$ where $tx=x$, $ty=-y$, $xy=z$, $zL=0$, $t^{[p]}=t$, $x^{[p]}=y^{[p]}=z^{[p]}=0$. By a direct computation, the only ideals of $L$ are $\{0\}$, $Fz$, $Fz+Fx$, $Fz+Fy$, $Fz+Fx+Fy$ and $L$.

\begin{enumerate}
\item[(a)] Suppose $S(z)=0$. Then $\rho(z)$ is nilpotent. Let $I=Fz+Fx+Fy$. Then $I$ is an ideal such that $I^{(1)}$ operates nilpotently on $V$. As $x^{[p]}=y^{[p]}=z^{[p]}=0$, $S(x)$, $S(y)$, $S(z)$ is the only eigenvalue for $\rho(x)$, $\rho(y)$, $\rho(z)$, respectively. This shows that $\lambda=S\mid_I$ is the only eigenvalue function of $I$.

\item[(a.1)] Suppose $S(I)=0$. This implies $I$ acts nilpotently on $V$ i.e., $I\cdot V=0$. By the irreducibility of $V$, $\rho|_I=0$. It follows that $V$ is an irreducible $L/I$-module. Thus, $V$ is one-dimensional. $V=Fv$ where $x \cdot v=0$, $y\cdot v=0$, $z\cdot v=0$, $h\cdot v=\alpha v$ for some $\alpha\in I$. This is exactly the same case as Example \ref{hxisoclasses}(a). For $S$, there are $p$ nonisomorphic irreducible modules determined by the eigenvalue of $\rho(h)$ and they are all one-dimensional.
\begin{align*}
\text{$V=Fv$, where $x \cdot v=0$, $y\cdot v=0$, $z\cdot v=0$, $h\cdot v=\alpha v$ for some $\alpha\in I$}\\
\text{where $\alpha$ is a solution of $X^p-X={S(h)}^p$}.
\end{align*}

\item[(a.2)] Suppose $S(I)\neq 0$. Then $L^{\lambda}\neq L$. Then $L^{\lambda}$ is a subalgebra containing a maximal ideal $I$ and thus $L^{\lambda}=I$. As $V^{\lambda}$ is $I$-irreducible, $\dim_FV^{\lambda}=1$. 
\begin{align*}
\text{$V^{\lambda}=Fv$, where $z\cdot v=0$, $x \cdot v=S(x)v$, $y\cdot v =S(y)v$.}
\end{align*}
Note that $V^{\lambda}$ is uniquely determined by $S$. By Corollary \ref{VinducVlambisoandchar}, $V$ is induced by $V^{\lambda}$, so $V$ is also uniquely determined by $S$. By Proposition \ref{dimind}, $\dim_FV=p$.
\begin{align*}
\text{$V=\Ind^L_{L^{\lambda}}(V^{\lambda}, S)$}& \text{ where $L^{\lambda}=Fx+Fy+Fz$},\\
&\text{$V^{\lambda}=Fv$ such that $z\cdot v=0$, $x \cdot v=S(x)v$, $y\cdot v =S(y)v$.}
\end{align*}

\item[(b)] Suppose $S(z)\neq0$. Let $I=Fz+Fx$. Then $I$ is an ideal such that $I^{(1)}$ operates nilpotently on $V$. As $x^{[p]}=z^{[p]}=0$, $S(x)$ and $S(z)$ is the only eigenvalue for $\rho(x)$ and $\rho(z)$, respectively. This shows that $\lambda=S\mid_I$ is the only eigenvalue function of $I$. $L^{\lambda}=Fz+Fx+F(t+S(x)+{S(z)}^{-1}y)$. By Corollary \ref{VinducVlambisoandchar} and Theorem \ref{ASlambisoBSlamb}, any irreducible $L$-module $V$ with character $S$ has a corresponding irreducible $L^{\lambda}$-module with character $S\mid_{L^{\lambda}}$ and vice versa. Moreover, $\dim_FV=p\dim_FV^{\lambda}$. To find all irreducible $L$-module $V$, we first find all possible $V^{\lambda}$. Define $h:=t+S(x){S(z)}^{-1}y$ and $x':=x-S(x){S(z)}^{-1}z$. Note that $S(x')=0$. $L^{\lambda}=Fh\oplus Fx'\oplus Fx$ such that $hx'=x'$. Since $z$ is central by the equation $zL=0$, $V^{\lambda}$ is irreducible and also is irreducible for $Fh\oplus Fx'$. Define $L'^{\lambda}:= Fh\oplus Fx'$ where $hx'=x'$, $h^{[p]}=h$ and $x'^{[p]}=0$. Then this is exactly the case (a) in Example \ref{hxisoclasses}. As $x'^{[p]}=0$ and $S(x')=0$, ${\rho(x')}^p=\rho(x'^{[p]})+{S(x')}^p\id_V=0$. $I:=Fx'$ is an ideal of $L$ which acts nilpotently on $V$. By the irreducibility of $V'$, $\rho(x')=0$. Then $V'$ is an irreducible $L/Fx'$-module and thus $V'$ is one-dimensional. $V'=Fv'$ where $x' \cdot v'=0$ and $h\cdot v'=\alpha v'$ for some $\alpha\in I$.

To determine $\alpha$, note that $\rho(h)=\alpha\id_V$. ${S(h)}^p\id_V={\rho(h)}^p-\rho(h^{[p]})={\rho(h)}^p-\rho(h)=(\alpha^p-\alpha)\id_V$. This implies $\alpha$ is a solution of the equation $X^p-X={S(h)}^p$ in $F$. For $S$, there are $p$ nonisomorphic irreducible $L'^{\lambda}$-modules of $V'$ determined by the eigenvalue of $\rho(h)$ and they are all one-dimensional.
\begin{align*}
V'=Fv,\;\;\quad\text{where} \; \quad &x' \cdot v=0,\quad h\cdot v=\alpha v \;\text{  for some }\alpha\in I\\
&\text{where $\alpha$ is a solution of $X^p-X={S(h)}^p$.}
\end{align*}
 For $S$, there are $p$ nonisomorphic irreducible $L^{\lambda}$-modules $V^{\lambda}$ determined by the eigenvalue of $\rho(h)$ and they are all one-dimensional.
\begin{align*}
V^{\lambda}=Fv,\;\;\quad\text{where} \; \quad x' \cdot v=0,\quad &h\cdot v=\alpha v  \quad z\cdot v=S(z)v\;\text{  for some }\alpha\in I\\
&\text{where $\alpha$ is a solution of $X^p-X={S(h)}^p$.}
\end{align*}
By Theorem \ref{ASlambisoBSlamb}, there are exactly $p$ nonisomorphic irreducible $L$-modules V. By Proposition \ref{dimind}, $\dim_FV=p$.
\begin{align*}
V=\Ind^L_{L^{\lambda}}&(V_{\alpha}, S) \\
&\text{where $L^{\lambda}=Fh\oplus Fx'\oplus Fx$, $V_{\alpha}=Fv$ such that $z\cdot v =S(z)v$,}\\
&\text{$x' \cdot v=S(x')v$, $(t+S(x){S(z)}^{-1}y)\cdot v =\alpha v$ }\\
&\text{where $\alpha$ is a solution of $X^p-X={S(h)}^p$. }
\end{align*}

\begin{Example}
Let $L=Fh\oplus Fx\oplus Fy$ where $hx=x$, $hy=\alpha y$ and $xy=0$. Need additional condition to make $L$ restricted, or more weakly, restrictable. 

\begin{enumerate}
\item[(a)] If $\alpha^p=\alpha$, then $L$ is restrictable. Then $h^{[p]}=h$ and $x^{[p]}=y^{[p]}=0$. Choose $I=Fx+Fy$. Then $I$ is an ideal such that $I^{(1)}$ operates nilpotently on $V$. As $x^{[p]}=y^{[p]}=0$, ${\rho(x)}^p=\rho(x^{[p]})+{S(x)}^p\id_V={S(x)}^p\id_V$. ${(\rho(x)-S(x)\id_V)}^p={\rho(x)}^p-{S(x)}^p\id_V=0$. This shows that $S(x)$ is the only eigenvalue of $\rho(x)$. Similarly, $S(y)$ is the only eigenvalue for $\rho(y)$. This implies that $\lambda=S|_I$ is the only eigenvalue function of $I$.
\item[(a.1)] Assume that $S(I)=0$. Then $I$ acts nilpotently on $V$, i.e., $I\cdot V=0$. By the irreducibility of $V$, $\rho|_I=0$. This shows that $V$ is an irreducible $L/I$-module. Since any irreducible abelian module is one-dimensional, $V$ is one-dimensional. Thus, $V=Fv$ where $x \cdot v=0$ and $h\cdot v=\beta v$ for some $\beta\in I$. By the equation $h\cdot v=\beta v$, $\rho(h)=\beta \id_V$. Note that $h^{[p]}=h$. By the definition of character, ${S(h)}^p\id_V={\rho(h)}^p-\rho(h^{[p]})={\rho(h)}^p-\rho(h)=(\beta^p-\beta)\id_V$ and this implies that $\beta$ is a solution of $X^p-X={S(h)}^p$. For $S$, there are $p$ nonisomorphic irreducible modules determined by the eigenvalue of $\rho(h)$ and they are all one-dimensional as $\dim_FV=1$.
\begin{align*}
V=Fv,\;\;\quad\text{where} \; \quad& x \cdot v=0,\quad h\cdot v=\beta v \;\text{  for some }\beta\in I\\
&\text{where $\beta$ is a solution of $X^p-X={S(h)}^p$.}
\end{align*}

\item[(a.2)] Assume that $S(I)\neq 0$. Then $L^{\lambda}\neq L$. As $L^{\lambda}$ is a subalgebra containing $I$, $L^{\lambda}=I$. Then $V^{\lambda}$ is $I$-irreducible, so $\dim_FV^{\lambda}=1$.
\begin{align*}
\text{$V^{\lambda}=Fv$ where $x\cdot v=S(x)v$, $y \cdot v=S(y)v$.}
\end{align*}
By Corollary \ref{VinducVlambisoandchar}, $V$ is induced by $V^{\lambda}$. Then $V$ is uniquely determined by $S$. By Proposition \ref{dimind}, $\dim_FV=p$.
\begin{align*}
\text{$V=\Ind^L_{L^{\lambda}}(V^{\lambda}, S)$}& \text{ where $L^{\lambda}=Fx+Fy$},\\
&\text{$V^{\lambda}=Fv$ such that $x \cdot v=S(x)v$, $y\cdot v =S(y)v$.}
\end{align*}

\item[(b)]
Assume that $L$ is nonrestrictable. Then $\alpha^p\neq \alpha$. Embed $L$ to a four-dimensional algebra $L':=L\oplus Ft$ where $hx=x$, $hy=\alpha y$, $xy=0$, $ht=0$, $tx=x$, and $ty=\alpha^py$, $x^{[p]}:=y^{[p]}:=0$, $h^{[p]}:=t$ and $(h-t)^{[p]}:={(\alpha-\alpha^p)}^{p-1}(h-t)$. Then $L'$ is restricted. Moreover, $L'$ is the minimal $p$-envelope. As $ht=0$, $t^{[p]}=h^{[p]}-{(\alpha-\alpha^p)}^{p-1}(h-t)=[{(\alpha-\alpha^p)}^{p-1}+1]t-{(\alpha-\alpha^p)}^{p-1}h$.

Let $\rho: L \rightarrow \gl(V)$ be an irreducible representation of $L$. By Theorem \ref{existextendreplsubmoduleGsubmodule}, there exists an irreducible representation $\rho': L' \rightarrow \gl(V)$ of $L'$. Let $S'$ be a character of $\rho'$. As ${\rho'(h)}^p-\rho'(h^{[p]})={S'(h)}^p\id_V$, ${\rho'(h)}^p-\rho'(t)={S'(h)}^p\id_V$. Define $\rho''$ such that $\rho''|_L=\rho$ and $\rho''(t):=\rho'(t)+{S'(h)}^p\id_V$. Then $\rho''$ is another extension of $\rho$ whose character $S''$ vanishes on $h$. Then the irreducible representations of $L$ correspond uniquely to the irreducible $S''$-representations of $L'$ which satisfy $S''(h)=0$. Therefore, we determine the irreducible $S''$-representations of $L'$ such that $S''(h)=0$.

\item[(b.1)] Assume $S''(x)=S''(y)=0$. Let $I=Fx+Fy$. Then $I$ is an ideal such that $I^{(1)}$ operates nilpotently on $V$. Note that $S''(I)=0$. This implies $I$ acts nilpotently on $V$, i.e., $I\cdot V=0$. By the irreducibility of $V$, $\rho''|_I=0$. This implies that $V$ is an irreducible $L/I$-module. As $L/I$ is abelian and $V$ is irreducible, $V$ is one-dimensional.
\begin{align*}
\text{$V=Fv$, where $x \cdot v=0$, $y\cdot v=0$, $h\cdot v=0$, $t\cdot v=\beta v$ for some $\beta\in I$}.
\end{align*}
${S''(t)}^p\id_V={\rho''(t)}^p-\rho''(t^{[p]})={\rho''(t)}^p-\rho''(t)=(\beta^p-\beta)\id_V$. This implies $\beta$ is a solution of the equation $X^p-X={S''(t)}^p$. For $S''$, there are $p$ nonisomorphic irreducible $L$-modules of $V$ determined by the eigenvalue of $\rho''(t)$ and they are all one-dimensional.
\begin{align*}
\text{$V=Fv$, where $x \cdot v=0$, $y\cdot v=0$, $h\cdot v=0$, $t\cdot v=\beta v$ for some $\beta\in I$}.\\
\text{where $\beta$ is a solution of $X^p-X={S''(t)}^p$.}
\end{align*}

\item[(b.2)]
Assume $S''(x)\neq0$ and $S''(y)=0$. Then $L^{\lambda}\neq L$. $L^{\lambda}=Fx\oplus Fy\oplus F(h-t)$. Noe that $V^{\lambda}$ is irreducible. As $x$ is central by the equations $xy=0$ and $x(h-t)=0$, $V^{\lambda}$ is irreducible for $Fy\oplus F(h-t)$.

Define $L'^{\lambda}:= Fy\oplus F(h-t)$. Let $I:=Fy$. Then $I$ is an ideal of $L$ which acts nilpotently on $V'^{\lambda}$. By the irreducibility of $V'^{\lambda}$, $\rho''(y)=0$. This implies $V'^{\lambda}$ is an irreducible $L/Fy$-module, so $V'^{\lambda}$ is one-dimensional.
\begin{align*}
\text{$V'^{\lambda}=Fv$, where $y\cdot v=0$, $(h-t)\cdot v=\gamma v$ for some $\gamma \in I$}.
\end{align*}
We want to specify the condition of $\gamma$. $S''(h-t)^p\id_{V'}=\rho''(h-t)^p-\rho''((h-t)^{[p]})=\rho(h-t)^p-\rho((\alpha-\alpha^p)^{p-1}(h-t))=\rho(h-t)^p-(\alpha-\alpha^p)^{p-1}\rho(h-t)=(\gamma^p-(\alpha-\alpha^p)^{p-1}\gamma)\id_{V'}$. This implies $\gamma$ is a solution of $X^p-(\alpha-\alpha^p)^{p-1}X=S''(h-t)^p$. For $S''$, there are $p$ nonisomorphic irreducible $L'^{\lambda}$-modules of $V'$ determined by the eigenvalue of $\rho''(h-t)$ and they are all one-dimensional.
\begin{align*}
\text{$V'=Fv$, where $y \cdot v=0$, $(h-t)\cdot v=\gamma v$ for some $\gamma \in I$}\\
\text{where $\gamma$ is a solution of $X^p-(\alpha-\alpha^p)^{p-1}X={S''(h-t)}^p$.}
\end{align*}
For $S'$, there are $p$ nonisomorphic irreducible $L^{\lambda}$-modules of $V^{\lambda}$ determined by the eigenvalue of $\rho''(h-t)$ and they are all one-dimensional.
\begin{align*}
\text{$V^{\lambda}=Fv$, where $x \cdot v=S''(x)v$, $y \cdot v=0$, $(h-t)\cdot v=\gamma v$ for some $\gamma \in I$}\\
\text{where $\gamma$ is a solution of $X^p-(\alpha-\alpha^p)^{p-1}X={S''(h-t)}^p$.}
\end{align*}
Then the irreducible $L'$-modules $V$ are
\begin{align*}
\text{$V=\Ind^{L'}_{L^{\lambda}}(V^{\lambda}, S'')$ where $L^{\lambda}=Fx+Fy+F(h-t)$},\\
\text{$V^{\lambda}=Fv$ such that $x \cdot v=S''(x)v$, $y \cdot v=0$, $(h-t)\cdot v=\gamma v$ for some $\gamma \in I$}\\
\text{where $\gamma$ is a solution of $X^p-(\alpha-\alpha^p)^{p-1}X={S''(h-t)}^p$.}
\end{align*}

\item[(c.3)]Assume $S''(x)=0$ and $S''(y)\neq0$. This is similar to the case (b.2).
\begin{align*}
\text{$V^{\lambda}=Fv$, where $x \cdot v=0$, $y \cdot v=S''(y) v$, $(h-t)\cdot v=\gamma v$ for some $\gamma \in I$}\\
\text{where $\gamma$ is a solution of $X^p-(\alpha-\alpha^p)^{p-1}X={S''(h-t)}^p$.}
\end{align*}

Then the irreducible $L'$-modules $V$ are
\begin{align*}
\text{$V=\Ind^{L'}_{L^{\lambda}}(V^{\lambda}, S'')$ where $L^{\lambda}=Fx+Fy+F(h-t)$},\\
\text{$V^{\lambda}=Fv$ such that $x \cdot v=0$, $y \cdot v=S''(y) v$, $(h-t)\cdot v=\gamma v$ for some $\gamma \in I$}\\
\text{where $\gamma$ is a solution of $X^p-(\alpha-\alpha^p)^{p-1}X={S''(h-t)}^p$.}
\end{align*}
\item[(b.4)] Assume $S''(x)\neq 0$ and $S''(y)\neq0$. Let $I=Fx+Fy$. Then $L^{\lambda}\neq L$. As $L^{\lambda}$ is a subalgebra containing $I$, $L^{\lambda}=I$. Then $V^{\lambda}$ is $I$-irreducible, and thus $\dim_FV^{\lambda}=1$.
 \begin{align*}
\text{$V^{\lambda}=Fv$ where $x\cdot v=S''(x)v$, $y \cdot v=S''(y)v$.}
\end{align*}
As $V=\Ind^L_{L^{\lambda}}(V^{\lambda}, S)$, $V$ is uniquely determined by $S''$. In other words, there is exactly one irreducible $L$-module. Moreover, $\dim_FV=p^{\dim_FL'/{L'}^{\lambda}}\cdot \dim_FV^{\lambda}=p^2\cdot 1=p^2$.
\begin{align*}
\text{$V=\Ind^{L'}_{L'^{\lambda}}(V^{\lambda}, S'')$ where ${L'}^{\lambda}=Fx+Fy$},\\
\text{$V^{\lambda}=Fv$ such that $x \cdot v=S''(x)v$, $y \cdot v=S''(y)v$}\\
\end{align*}
\end{enumerate}
\end{Example}

\end{enumerate}
\end{Example}

\chapter{Conclusions}

Lie algebras over fields of positive characteristic behave differently from those over fields of characteristic zero. With the introduction of the $p$-mapping, modular Lie algebras possess $p$-envelopes and restricted enveloping algebras. In particular, the structure of a restricted enveloping algebra depends on the $p$-mapping of the given Lie algebra. Modular representations admit an invariant, the \emph{character} $S$, and the set of induced modules with a given character corresponds to a set of submodules of the original modules.

If a given modular Lie algebra is not restricted, we consider its minimal $p$-envelope. By determining all irreducible modules of this minimal $p$-envelope, and using the one-to-one correspondence between induced modules and the original modules, we can obtain all irreducible modules of the given Lie algebra.

The central question arising from this project is how one might develop a more general notion of character that applies to a broader class of modular Lie algebras.

\bibliography{mainproject}{}
\bibliographystyle{plain}

% Comment the following THREE lines if you do NOT have an Appendix
\appendix

\chapter{Multiplication Tables for Given Examples}

\section{Multiplication Table of $L$}\label{chartabofC3}

\begin{table}[h]
\centering
\begin{tabular}{ c | c  c c  }
\hline
 \hline
  $$   & $e$ & $f$ & $h$\\
 \hline
 $e$   &     &$h$&   -2$e$\\
$f$&     &   &2$f$\\
$h$& &  &   \\
 \hline
 \hline
\end{tabular}
\caption{$\Sl(2,F)$}
\end{table}

\begin{table}[h]
\centering
\begin{tabular}{ c | c  c c c c  }
\hline
 \hline
  $$   & $e$& $f$ & $h$ & $u$ & $v$\\
 \hline
 $e$   & &$h$   & $-2e$&$u$& $-v$  \\
$f$&     & & 2$f$ & 0 & $u$ \\
$h$&&  &&$v$  & 0  \\
$u$&&&&&0 \\
$v$&&&&&\\
 \hline
 \hline
\end{tabular}
\caption{$L=Fe\oplus Fh\oplus Ff\oplus Fu\oplus Fv$}
\end{table}

\begin{table}[h]
\centering
\begin{tabular}{ c | c  c   }
\hline
 \hline
  $$   & $h$ & $x$ \\
 \hline
 $h$   &     &$x$\\
$x$     &    &\\
 \hline
 \hline
\end{tabular}
\caption{$L=Fh\oplus Fx$}
\end{table}

\begin{table}[H]
\centering
\begin{tabular}{ c | c  c c c   }
\hline
 \hline
  $$   & $t$& $x$ & $y$ & $z$ \\
 \hline
 $t$   & &$x$   & $-y$&0\\
$x$&     & & $z$ & 0 \\
$y$&&  &&0   \\
$z$&&&&\\
 \hline
 \hline
\end{tabular}
\caption{$L=Ft\oplus Fx\oplus Fy\oplus Fz$}
\end{table}

\begin{table}[h]
\centering
\begin{tabular}{ c | c  c c  }
\hline
 \hline
  $$   & $h$ & $x$ & $y$\\
 \hline
 $h$   &     &$x$&   $\alpha$\\
$x$&     &   &0\\
$y$& &  &   \\
 \hline
 \hline
\end{tabular}
\caption{$L=Fh\oplus Fx\oplus Fy$}
\end{table}

\begin{table}[h]
\centering
\begin{tabular}{ c | c  c c  }
\hline
 \hline
  $$   & $e$ & $f$ & $h$\\
 \hline
 $e$   &     &$h$&   $e$\\
$f$&     &   &$f$\\
$h$& &  &   \\
 \hline
 \hline
\end{tabular}
\caption{$\fsl(2,F)$}
\end{table}

\begin{table}[h]
\centering
\begin{tabular}{ c | c  c c c c  }
\hline
 \hline
  $$   & $e$& $f$ & $h$ & $u$ & $v$\\
 \hline
 $e$   & &$h$   & $e$&0& $f$  \\
$f$&     & & $f$ & $e$ & 0 \\
$h$&&  &&0  & 0  \\
$u$&&&&&0 \\
$v$&&&&&\\
 \hline
 \hline
\end{tabular}
\caption{The minimal $p$-envelope of $\fsl(2,F)$}
\end{table}

\newpage
\printindex

\end{document}